\documentclass[reqno,11pt]{amsart}

\usepackage{mathrsfs}
\usepackage{amsmath,amssymb,slashed}
\usepackage{color,xcolor}
\usepackage[english]{babel}
\usepackage[utf8]{inputenc}
\usepackage{amsmath,amscd}
\usepackage{amsfonts}
\usepackage{amsthm}
\usepackage{graphicx}
\usepackage[colorinlistoftodos]{todonotes}
\usepackage{amsmath,amssymb,amsfonts}
\usepackage{underscore}
\usepackage{xy}
\usepackage[T1]{fontenc}
\usepackage[utf8]{inputenc}
\usepackage{amsthm}
\usepackage{mathrsfs}
\usepackage{amsmath,amssymb,slashed}
\usepackage{color,xcolor}
\usepackage{amsmath,amscd}

\input xy
\xyoption{all}

\theoremstyle{plain}
\newtheorem{thm}{Theorem}[section]
\newtheorem{prop}[thm]{Proposition}
\newtheorem{lem}[thm]{Lemma}

\newtheorem{pro}[thm]{Problem}

\theoremstyle{definition}
\newtheorem{defn}[thm]{Definition}

\theoremstyle{remark}
\newtheorem{rmk}[thm]{Remark}

\newcommand{\B}{\mathcal{B}}

\newcommand{\la}{\langle}
\newcommand{\ra}{\rangle}

\makeatletter
\@namedef{subjclassname@2020}{\textup{2020} Mathematics Subject Classification}
\makeatother

\begin{document}

\title[]{Takesaki duality for weak* closed $L^p$-operator crossed products}

\author[Z. Wang]{Zhen Wang}
%\curraddr{Department of Mathematics\\Jilin University\\Changchun 130012\\P.~R. China}
\address{School of Mathematics\\HangZhou Normal University \\HangZhou 311121\\
P. R. China}
\email{wangzhen@hznu.edu.cn}

%\author[S. Zhu]{Sen Zhu}
%\address{Department of Mathematics\\Jilin University\\Changchun 130012\\P.~R. China}
%\email{zhusen@jlu.edu.cn}
%\thanks{}


\begin{abstract}
The aim of this paper is to study  Takesaki duality for weak* closed $L^p$-operator crossed product $W^*_p(G,A,\alpha)$, where $G$ is a countable discrete Abelian group, $A$ is a unital separable weak* closed $L^p$-operator algebra ($p>1$), and $\alpha$ is a weak* continuous $p$-completely isometric action of $G$ on $A$. In this paper, we construct a weak* continuous homomorphism  $\Phi$ from $W^*_p(\hat{G},W^*_p(G,A,\alpha),\hat{\alpha})$ to $\mathcal{B}(l^{p}(G))\bar{\otimes}A$. We show that $\Phi$ is an isomorphism if and only if either $p=2$ or $G$ is finite, and $\Phi$ is an isometric isomorphism if either $p=2$ or $G$ is trivial.  It is also proved that $\Phi$ is equivariant for the double dual action $\hat{\hat{\alpha}}$ of $G$ on $W^*_p(\hat{G},W^*_p(G,A,\alpha),\hat{\alpha})$ and the action $\mathrm{Ad}\rho_p\otimes\alpha$ of $G$ on $\mathcal{B}(l^p(G))\bar{\otimes} A$.
Furthermore, we prove that  $W^*_p(\hat{G},W^*_p(G,A,\alpha),\hat{\alpha})$ is weak* continuous isometrically isomorphic to $\mathcal{B}(l^{p}(G))\bar{\otimes}A$ if and only if either $p=2$ or $G$ is trivial, and $W^*_p(\hat{G},W^*_p(G,A,\alpha),\hat{\alpha})$ is weak* continuous isomorphic to $\mathcal{B}(l^{p}(G))\bar{\otimes}A$ if and only if either $p=2$ or $G$ is finite when $A=M_n^p$. This shows that  Takesaki duality theorem of von Neumann algebras can be generalized to weak* closed $L^2$-operator algebras, and this theorem can not be generalized to  weak* closed $L^p$-operator algebras when $p\in (1,\infty)\setminus\{2\}$.
\end{abstract}

\date{\today}
\maketitle

 %\tableofcontents

\section{Introduction}

The Takesaki duality theorem for crossed products of von Neumann algebras, pioneered by M. Takesaki in 1973, represents a profound symmetry inherent to von Neumann algebras equipped with group actions.
Recall that a W*-dynamical system is a triple $(G,M,\alpha)$ consisting of a von Neumann algebra $M$, a locally compact group $G$, and a continuous homomorphism $\alpha:G\rightarrow \mathrm{Aut}(M)$, where $\mathrm{Aut}(M)$ is the group of $*$-automorphisms of $M$ with pointwise weak* topology. We denote by $M\rtimes_\alpha G$ the crossed product of $M$  by the action $\alpha$ of $G$ on $M$.
If, in addition, $G$ is Abelian, then there exists a dual action $\hat{\alpha}$ of $\hat{G}$ on $M\rtimes_\alpha G$. The Takesaki duality theorem (\cite[Theorem4.5]{dual}) says that the iterated crossed product $(M\rtimes_{\alpha}G)\rtimes _{\hat{\alpha}}\hat{G}$ is $*$-isomorphic to $\mathcal{B}(L^2(G))\bar{\otimes} M$, where $\mathcal{B}(L^2(G))$ is the algebra of all bounded linear operators on $L^2(G)$ and $\bar{\otimes}$ is the von Neumann algebraic tensor product. This theorem not only deepened the understanding of the structure of von Neumann algebras but also became an indispensable tool in their classification (see \cite{Takesaki}).

%It was subsequently generalized by Nakagami and Sutherland to twisted crossed products and regular extensions, and later integrated into the broader framework of coactions by Yamanouchi and others.

%Building on the foundational Tomita-Takesaki modular theory \cite{takesaki2002theory}, Takesaki demonstrated that the crossed product construction is, in a precise sense, an involution up to stabilization.

%In 1973, M. Takesaki obtained a remarkable duality theorem for crossed products of von-Neumann algebras.
%Takesaki duality theorem is a fundamental  result concerning crossed products of von-Neumann algebras.  The Takesaki duality theorem has applications in the structure analysis of von Neumann algebra of type $\mathrm{III}$ (see \cite{Takesaki}), and the proof of  injectivity implying semidiscreteness (see \cite[Theorem 9.3.3]{Brown and Ozawa}).

Inspired by Takesaki duality theorem for crossed products of von Neumann algebras, H. Takai proved a duality theorem for crossed products of $C^*$-algebras in 1975. The Takai duality theorem says that the iterated crossed product $(A\rtimes_{\alpha}G)\rtimes _{\hat{\alpha}}\hat{G}$ is $*$-isomorphic to $\mathcal{K}(L^{2}(G))\otimes A$,
where $\mathcal{K}(L^{2}(G))$ denotes the algebra of compact operators on $L^{2}(G)$.

In 2014, N. C. Phillips asked a problem if there is an $L^p$-version of the Takai duality theorem (see \cite[Problem 8.7]{PlpsOpenQues}). S. Zhu and I proved that the Takai duality theorem can not be generalized to crossed products of $L^p$-operator algebras when $p\neq 2$ (see \cite{WZ}).

For $p\in[1,\infty)$, we say that a Banach algebra $A$ is an {\it $L^{p}$-operator algebra} if
it is isometrically isomorphic to a norm closed subalgebra of the algebra $\mathcal{B}(E)$ of all bounded linear operators on some $L^p$-space $E$.
The study of $L^p$-operator algebras traces back to C. Herz's influential work on harmonic analysis of group algebras in the 1970's \cite{Herz,Herz1,Herz2}.
Recently, people renew interest in $L^p$-operator algebras due to the work of N. C. Phillips.
In the passing ten years, N. C. Phillips  introduced and studied $L^p$-operator algebras \cite{Phillips Lp Cuntz,Odp,Lp UHF,N. C. Phillips Lp, PlpsOpenQues,Phillips look like,Lp AF}. These studies encourage many authors to participate in the research of $L^p$-operator algebras. This includes the work on group $L^p$-operator algebras \cite{Gardella and Thiel,Z}; groupoid $L^p$-operator algebras \cite{Gardella and Lupini groupoid}; $L^p$-Roe type algebras \cite{Chung}; $L^p$-operator crossed products \cite{Gardella and Thiel convolution,WZ,WZ2} and the $l^p$-Toeplitz algebra \cite{WW}.

Since von-Neumann algebras are weak* closed $L^2$-operator algebras, it is natural to ask if there is a  Takesaki duality theorem for crossed products of weak* closed $L^p$-operator algebras.
The aim of this paper is study this problem.

Let $(G,A,\alpha)$ be a weak* closed $L^p$-operator algebra dynamical system, we will introduce the weak* closed $L^p$-operator crossed product $W^*_p(G,A,\alpha)$ of $A$ by the action of $\alpha$ (see Section \ref{crossed}). If $G$ is Abelian, then there exists a dual action $\hat{\alpha}$ of $\hat{G}$ on $W^*_p(G,A,\alpha)$ (see Section \ref{dual}).
It is natural to consider the following problem.

\begin{pro}
Is it the iterated weak* closed $L^p$-operator crossed product $W^*_p(\hat{G},W^*_p(G,A,\alpha),\hat{\alpha})$ weak* continuous isometric isomorphic (or isomorphic) to $\mathcal{B}(L^p(G))\bar{\otimes} A$.
\end{pro}

We will study this problem when $G$ is a countable discrete Abelian group, $E$ is a $\sigma$-finite separable $L^p$-space, $A\subset\mathcal{B}(E)$ is an unital weak* closed subalgebra such that the unit $I_A$ of $A$ is the identity operator on $E$. The following is main theorems of this paper.

\begin{thm}\label{Thm1}
Let $p\in (1,\infty)$. Let $(G,A,\alpha)$ and $\hat{\alpha}$ as above. If $\alpha$ is a $p$-completely isometric action of $G$ on $A$, then there is a homomorphism $\Phi:W^*_p(\hat{G},W^*_p(G,A,\alpha),\hat{\alpha})\rightarrow \mathcal{B}(l^p(G))\bar{\otimes} A$. Moreover,
\begin{enumerate}
\item[(i)] the map $\Phi$ is injective contractive weak* continuous ;
\item[(ii)] the map $\Phi$ is surjective if and only if $G$ is finite or $p=2$;
\item[(iii)] the map $\Phi$ is  isometrically  isomorphic if either $p=2$ or $G$ is trivial;
\item[(iv)] the map $\Phi$ is equivariant for the double dual action $\hat{\hat{\alpha}}$ of $G$ on $W^*_p(\hat{G},W^*_p(G,A,\alpha),\hat{\alpha})$ and the action $\mathrm{Ad}\rho\otimes\alpha$ of $G$ on $\mathcal{B}(l^p(G))\bar{\otimes}A$, where $$\hat{\hat{\alpha}}_{t}(F)(\gamma, s):=\overline{\gamma(t)} F(\gamma, s)$$ for all $t\in G, F\in C_c(\hat{G}\times G,A)\subset W^*_p(\hat{G},W^*_p(G,A,\alpha),\hat{\alpha})$, and
   $$(\mathrm{Ad}\rho_p\otimes\alpha)_s (T\otimes a)=(\mathrm{Ad}\rho_p)_s(T)\otimes \alpha_s(a)$$ for all $T\in \mathcal{B}(l^p(G)), a\in A$.
\end{enumerate}

\end{thm}

\begin{thm}\label{Thm2}
Let $p\in (1,\infty)$. Then $W^*_p(\hat{G},W^*_p(G,A,\alpha),\hat{\alpha})$ is weak* continuous isometric isomorphic to $\mathcal{B}(l^p(G))\bar{\otimes} A$ if and only if either $p=2$ or $G$ is trivial.
\end{thm}

\begin{thm}\label{Thm3}
Let $p\in (1,\infty)$ and let $A=M_n^p$. Then $W^*_p(\hat{G},W^*_p(G,A,\alpha),\hat{\alpha})$ is weak* continuous isomorphic to $\mathcal{B}(l^p(G))\bar{\otimes} A$ if and only if either $p=2$ or $G$ is finite.
\end{thm}

\begin{rmk} This shows that  Takesaki duality theorem of von-Neumann algebras can be generalized to weak* closed $L^2$-operator algebras, but this theorem can not be generalized to  weak* closed $L^p$-operator algebras when $p\in (1,\infty)\setminus\{2\}$.
\end{rmk}

The paper is organized as follows.
In Section 2, we introduced crossed products of weak* closed $L^p$-operator algebras, and we give some basic properties on this crossed products.
In Section 3, we define the dual action $\hat{\alpha}$ of $\widehat{G}$ on $W^*_p(G,A,\alpha)$. In Section 4, we study $C^*$-cores of weak*closed $L^p$-operator algebras. This will have important applications in proving Theorem \ref{Thm2}. In Section 5, we explore the Gelfand transform of $p$-pseudomeasure algebra $PM_p(G)$ when $G$ is a locally compact abelian group. In Section 6, we give the proof of Theorem \ref{Thm1}. In Section 7, we will prove Theorem \ref{Thm2} and Theorem \ref{Thm3}.

In the end of this section, we recall some basic notations. We always let $p\in (1,\infty)$ and denoted by $q$ be its conjugate exponent
of $p$, that is, $\frac{1}{p}+\frac{1}{q}=1$. If $E$ is an $L^p$-space, then we let $E'$ be its dual space, and we let $\la .,.\ra:E'\times E\rightarrow \mathbb{C}$ be the dual pair.
If $T:E\rightarrow F$ is a bounded linear map between two $L^p$-spaces, then we denote by $T':F'\rightarrow E'$ be its dual operator.

\section{Crossed products of weak* closed $L^p$-operator algebras}\label{crossed}
In this section, we introduce crossed products of weak* closed $L^p$-operator algebras and their basis properties.

Let $E$ be a $\sigma$-finite separable $L^p$-space and $E'$ be the dual space of $E$.
Let $\mathcal{N}(E)=E'\widehat{\otimes}E$ denote the space of nuclear operator operators on $E$, where $\widehat{\otimes}$ is the projective tensor product. Then $\mathcal{B}(E)$ is the dual space of $\mathcal{N}(E)$ by way of dual paring $$\la T,\xi\otimes \eta\ra=\la\xi,T\eta\ra$$ where $\xi\in E'$ and $\eta\in E$.
We say that a net $(T_\alpha)$ in $\mathcal{B}(E)$ converges weak* to an operator $T$ in $\mathcal{B}(E)$ if $$\la T_\alpha,\omega\ra\rightarrow \la T,\omega\ra $$ for all $\omega\in E'\widehat{\otimes} E$.

A {\it weak* closed $L^p$-operator algebra $A$} is a subalgebra of $\B(E)$ which is weak* closed in $\mathcal{B}(E)$.
Let $\mathrm{Aut}(A)$ be the group of isometric automorphisms of $A$. We consider the pointwise weak* topology on $\mathrm{Aut}(A)$, that is, a net $\varphi_\lambda\in \mathrm{Aut}(A)$ converges to $\varphi_0\in \mathrm{Aut}(A)$ if $$\la\varphi_\lambda(a)-\varphi_0(a),\omega\ra\rightarrow 0$$ for all $a\in A$ and $\omega\in \mathcal{N}(E)$. A {\it weak* closed $L^p$-operator algebra dynamical system} is a triple $(G,A,\alpha)$ consisting of a locally compact group $G$, a weak* closed $L^p$-operator algebra $A$, and a weak* continuous isometric action of $G$ on $A$, that is, $\|\alpha_t(a)\|=\|a\|$ and $\alpha:G\rightarrow \mathrm{Aut}(A)$ is continuous, which means that for each $a\in A$, the function $$\alpha(a):G\rightarrow A;t\mapsto\alpha_t(a)$$ is continuous on $G$.

Let $G$ be a locally compact group, then there exists a left Haar measure $\mu$ on $G$.
Denote by $C_{c}(G,A,\alpha)$ the vector space of continuous compactly supported functions $G\rightarrow A$, made into an algebra over $\mathbb{C}$ with product given by twisted convolution, that is, $$(f*g)(t):=\int_{G}f(s)\alpha_{s}(g(s^{-1}t))d\mu(s)$$ for all $f,g\in C_{c}(G,A,\alpha)$ and $t\in G$.

Given two measure spaces $(X,\mu)$ and $(Y,\nu)$, we denote by $L^p(X,\mu)\otimes_p L^p(Y,\nu)$ the $L^p$-tensor product, which can be identified with $L^p(X\times Y,\mu\times\nu)$ via $\xi\otimes\eta(x,y)=\xi(x)\eta(y)$ for $\xi\in L^p(X,\mu)$ and $\eta\in L^p(Y,\nu)$.

Let $A\subset\mathcal{B}(E)$ be a weak* closed subalgebra, and
we denote by $\pi_0:A\rightarrow \mathcal{B}(E)$ be the canonical representation of $A$. Its associated {\it regular covariant representation} is the pair $(\pi,\lambda_p^E)$ given by
$$\pi(a)(\xi)(s):=\pi_{0}\left(\alpha_{s^{-1}}(a)\right)(\xi(s)) \quad$$ and $$\lambda_p^{E}(s)(\xi)(t):=\xi(s^{-1}t)$$
for all $a\in A$, $\xi\in L^{p}(G,E)\cong L^p(G)\otimes_p E$, and $s,t\in G$.
For any $t\in G$ and $a\in A$, one can check that $$\lambda_p^E(t)\pi(a)\lambda_p^E(t^{-1})=\pi(\alpha_t(a)).$$

The {\it integrated form} of $(\pi,\lambda_p^E)$ is an isometric representation $\pi\rtimes\lambda_p^E:C_c(G,A,\alpha)\rightarrow \mathcal{B}(L^p(G)\otimes_p E))$ given by $$\pi\rtimes\lambda_p^E(f):=\int_G \pi(f(t))\lambda_p^E(t)d\mu(t)$$ for all $f\in C_c(G,A,\alpha)$.

\begin{defn}
Let $A\subset \B(E)$ be a weak* closed subalgebra and let $\alpha$ be a weak* continuous isometric action of $G$ on $A$. The weak* closed $L^p$-operator crossed product of $A$ by the action of $\alpha$ is the weak* closure of the algebra generated by the operators $\pi\rtimes\lambda_p^E(f)$ in $\mathcal{B}(L^p(G)\otimes_p E))$ for all $f\in C_c(G,A,\alpha)$, and it is denoted by $W^*_p(G,A,\alpha)$.
\end{defn}

\begin{rmk}
Since the crossed products of von-Neumann algebras were used to denote by $W^*(M,G)$ (see \cite{van}), we define the weak* closed $L^p$-operator crossed product by $W^*_p(G,A,\alpha)$. If $A\subset\mathcal{B}(H)$ is a von Neumann algebra, then $W^*_2(G,A,\alpha)=A\rtimes_\alpha G$, where $H$ is a Hilbert space.
\end{rmk}

The pseudomeasure algebras are special case of weak* closed $L^p$-operator crossed product.
We let $\lambda_p:G\rightarrow \mathcal{B}(L^p(G))$ and $\rho_p:G\rightarrow \mathcal{B}(L^p(G))$ denote the  left and right regular representation, respectively, which are give by:
$$\lambda_p(s)g(t)=g(s^{-1}t),~~\rho_p(s)g(t)=g(ts)$$
for all $g\in L^p(G)$ and $s,t\in G$.

\begin{defn}
\begin{enumerate}
\item[(i)] The reduced group $L^p$-operator algebra of $G$, denoted $F^p_\lambda(G)$, is the completion of $C_c(G)$ with respect to the norm $\|\lambda_p(f)\|.$
\item[(i)] The $p$-pseudomeasure algebra of $G$, denoted $PM_p(G)$, is the weak* closure of $F^p_\lambda(G)$ in $\mathcal{B}(L^p(G))$.
\item[(ii)] The $p$-convoluters of $G$, denoted $CV_p(G)$, is defined as follows: $$CV_p(G)=\{T\in \mathcal{B}(L^p(G)):T\rho_p(s)=\rho_p(s)T \text{ for all } s\in G\}$$
\end{enumerate}
\end{defn}
\begin{rmk}
One famous problem is to ask whether $PM_p(G)=CV_p(G)$ for all groups $G$. M. Daws and N. Spronk proved that $PM_p(G)=CV_p(G)$ if $G$ has the approximation property (see \cite[Theorem 1.1]{Daws}). Since abelian groups have the approximation property, it follows that  $PM_p(G)=CV_p(G)$ when  $G$ is abelian.
\end{rmk}
If $A=\mathbb{C}$, and let $\mathrm{id}$ be the trivial action of $G$ on $A$, then we have $W^*_p(G,\mathbb{C},\mathrm{id})=PM_p(G)$.
When $p=2$, we have that $PM_2(G)=CV_p(G)=L(G)$, where $L(G)$ is the group von Neumann algebra of $G$.

For each positive integer $n$, {denote $l_n^p=L^p(\{1,2,\cdots,n\},\nu)$, where $\nu$ is the counting measure on $\{1,2,\cdots,n\}$.
We denote $M_n^p=\mathcal{B}(l_n^p)$}.
Given a norm closed subalgebra $A$ of $\mathcal{B}(L^p(X,\mu))$, we denote by $M_n^p\otimes_p A$ the {\it $L^p$-matrix algebra}, that is, the Banach subalgebra of $\mathcal{B}(L^p(\{1,2,\cdots,n\}\times X,\nu\times\mu))$ generated by all $T\otimes a$'s for $T\in M_n^p$ and $a\in \mathcal{B}(L^p(X,\mu))$. Clearly, each element of $M_n^p\otimes_p A$ is of form $[a_{i,j}]_{1\leq i,j\leq n}$ with $a_{i,j}\in A$,
which is also written as $\sum_{i,j=1}^n e_{i,j}\otimes a_{i,j}$, where $\{e_{i,j}\}_{1\leq i,j\leq n}$ are the canonical matrix units of $M_n^p$.

\begin{defn} %\cite{Lp AF}
Let $A$ be a normed closed subalgebra of $\mathcal{B}\left(L^p(X,\mu)\right)$, $B$ be a closed subalgebra of $\mathcal{B}\left(L^p(Y,\nu)\right)$
and $\varphi$ be a linear map $\varphi: A\rightarrow B$.  We denote by $\varphi_n$ the map from $M_n^p\otimes_p A$ to $M_n^p\otimes_p B$ defined by $$\varphi_n\left(\sum_{i,j=1}^n e_{i,j}\otimes a_{i,j}\right)=\sum_{i,j=1}^n e_{i,j}\otimes \varphi(a_{i,j}) $$ for
$\sum_{i,j=1}^n e_{i,j}\otimes a_{i,j}\in M_n^p\otimes_p A$.
We denote $$\|\varphi\|_{pcb}=\sup_{n\in \mathbb{Z}_{>0}}\| \varphi_n\|.$$
We say that $\varphi$ is {\it $p$-completely bounded} if $\|\varphi\|_{pcb}\leq C$ for some positive constant $C$, say that $\varphi$ is {\it $p$-completely contractive} if $\|\varphi\|_{pcb}\leq 1$, and say that $\varphi$ is {\it $p$-completely isometric} if $\varphi_n$ is isometric for all positive integer  $n$.
\end{defn}

\begin{defn}[{see \cite{Daws}}]
Let $E_1, E_2$ be two $L^p$-spaces, and let $A\subset\mathcal{B}(E_1), B\subset\mathcal{B}(E_2)$ be two weak* closed subalgebra. We define $A\bar{\otimes}B$ to be the weak* closure of $A\otimes_{alg} B$ in $\mathcal{B}(E_1\otimes_p E_2)$.
\end{defn}

\begin{lem} \label{tensor}
Let $E_3$ be an $L^p$-space, and let $C\subset\mathcal{B}(E_3)$ be a weak* closed subalgebra.
 If $\varphi:A\rightarrow B$ is a $p$-completely contractive weak* continuous linear map, then there exists a $p$-completely contractive weak* continuous linear map $\hat{\varphi}:A\bar{\otimes}C\rightarrow B\bar{\otimes}C$ such that $\hat{\varphi}(a\otimes c)=\varphi(a)\otimes c$ for all $a\in A, c\in C$ and $\|\hat{\varphi}\|_{pcb}\leq\|\varphi\|_{pcb}$. If $\varphi$ is $p$-completely isometric weak* continuous, then so is $\hat{\varphi}$. Furthermore, if $\varphi$ is a homomorphism, then so is $\hat{\varphi}$.
\end{lem}

\begin{proof}
By \cite[Theorem 6.4]{Daws1}, it follows that $\hat{\varphi}$ is contractive weak* continuous. If we replace $C$ by $M_n^p\bar{\otimes} C$, this implies that $\hat{\varphi}$ is $p$-completely contractive. It is easy to check the rest of this lemma.
\end{proof}

\begin{lem}\label{tensor3}
Let $(G,A,\alpha)$ be a weak* $L^p$-operator algebra dynamical system, where $G$ is a countable discrete group, $A\subset \mathcal{B}(E)$ be a weak* closed subalgebra, and $\alpha:G\rightarrow \mathrm{Aut}(A)$. Then $W^*_p(G,A,\alpha)\subset \mathcal{B}(l^p(G)\bar{\otimes}A$.
\end{lem}
\begin{proof}
Let $\pi_0:A\rightarrow \mathcal{B}(E)$ is the canonical representation, and let $(\pi,\lambda_p^E)$ be its associated regular representation.
For each $t\in G$, from the definition of $\lambda_p^E(t)$, we have $$\lambda_p^E(t)=\lambda_p(t)\otimes I_E,$$ where $\lambda_p:G\rightarrow \mathcal{B}(l^p(G))$ is the left regular representation, and $I_E$ is the identity operator on $E$. This shows that $\lambda_p^E(t)\in \mathcal{B}(l^p(G)\bar{\otimes}A$.

Let $\delta_t$ be the canonical basis of $l^p(G)$, and let $e_{t,t}$ is the rank one projection in $\B(l^p(G))$ satisfying $e_{t,t}(\sum_{s\in G} a_s\delta_s)=a_t\delta_t$ for all $\sum_{s\in G} a_s\delta_s\in l^p(G)$.
Identifying $l^p(G)\otimes_p E$ with the $L^p$-direct sum $\bigoplus_{t\in G}E$, we may simply take the $L^p$-direct sum representation $$\pi (a)=\bigoplus_{t\in G}\pi_0(\alpha_{t^{-1}}(a))\in \mathcal{B}(\bigoplus_{t\in G}E).$$
In the spatial $L^p$-operator tensor product picture, we have $$\pi(a)=\sum_{t\in G}e_{t,t}\otimes \pi_0(\alpha_{t^{-1}}(a))\in \mathcal{B}(l^p(G))\bar{\otimes}A.$$ This finish the proof of this Lemma.
\end{proof}

\begin{lem}\label{induce}
Let $(G,A,\alpha_1)$ and $(G,B,\alpha_2)$ be two weak* $L^p$-operator algebra dynamical system, where $G$ is a countable discrete group, $A\subset\mathcal{B}(E_1),B\subset\mathcal{B}(E_2)$ are two weak* closed subalgebras, and $\alpha_1:G\rightarrow \mathrm{Aut}(A), \alpha_2:G\rightarrow \mathrm{Aut}(B)$ are two continuous homomorphism. If $\varphi:A\rightarrow B$ is a weak* continuous $p$-completely contractive equivariant homomophism for the action $\alpha_1$ of $G$ on $A$ and the action $\alpha_2$ of $G$ on $B$, then there is a weak* continuous contractive homomorphism $\varphi\rtimes\mathrm{id}:W^*_p(G,A,\alpha_1)\rightarrow W^*_p(G,B,\alpha_2)$. If, in addition, $\varphi$ is a $p$-completely isometric isomorphism, then $\varphi\rtimes\mathrm{id}:W^*_p(G,A,\alpha_1)\rightarrow W^*_p(G,B,\alpha_2)$ is a weak* continuous isometric isomorphism.
\end{lem}
\begin{proof}
Since $\varphi:A\rightarrow B$ is a a weak* continuous $p$-completely contractive homomorphism, it follows from Lemma \ref{tensor} that there exists a  weak* continuous $p$-completely contractive homomorphism $\hat{\varphi}:\B(l^p(G))\bar{\otimes}A\rightarrow \B(l^p(G))\bar{\otimes}B$.

By Lemma \ref{tensor3}, it suffices to prove that $\hat{\varphi}(W^*_p(G,A,\alpha_1)\subset W^*_p(G,B,\alpha_2)$.
Let $\pi_1:A\rightarrow \mathcal{B}(E_1), \pi_2:B\rightarrow \mathcal{B}(E_2)$ be two canonical representations, and let $(\widetilde{\pi_1},\lambda_p^{E_1}), (\widetilde{\pi_2},\lambda_p^{E_2})$ be these associated regular representation, respectively.
Notice that $\varphi$ is equivariant the action $\alpha_1$ of $G$ on $A$ and the action $\alpha_2$ of $G$ on $B$, then for each $\sum_{s\in G}a_s\delta_s\in C_c(G,A,\alpha_1)$, we have
$$\begin{aligned} \hat{\varphi}(\sum_{s\in G}\widetilde{\pi_1}(a_s)\lambda_p^{E_1}(s)) &=
\hat{\varphi}\left(\sum_{s\in G}\left(\sum_{t\in G} e_{t,t}\otimes \alpha_{1,t^{-1}}(a_s)\right)\left(\lambda_p(s)\otimes I_{E_1}\right)\right)\\
 &=\hat{\varphi}\left(\sum_{s\in G}\sum_{t\in G} e_{t,t}\lambda_p(s)\otimes \alpha_{1,t^{-1}}(a_s)\right)\\
 &=\sum_{s\in G}\sum_{t\in G} e_{t,t}\lambda_p(s)\otimes \varphi\circ\alpha_{1,t^{-1}}(a_s)\\
 &=\sum_{s\in G}\sum_{t\in G} e_{t,t}\lambda_p(s)\otimes \alpha_{2,t^{-1}}\circ\varphi(a_s)\\
 &=\sum_{s\in G}\left(\sum_{t\in G}e_{t,t}\otimes\alpha_{2,t^{-1}}\circ\varphi(a_s)\right)\left(\lambda_p(s)\otimes I_{E_2}\right)\\
 &=\sum_{s\in G}\widetilde{\pi_2}(\varphi(s_s))\lambda_p^{E_2}(s)\in W^*_p(G,B,\alpha_2).\end{aligned}$$
Since $\hat{\varphi}$ is weak* continuous and $C_c(G,A,\alpha_1)$ is dense in $W^*_p(G,A,\alpha_1)$ in weak* topology, it follows that  $\hat{\varphi}(W^*_p(G,A,\alpha_1)\subset W^*_p(G,B,\alpha_2)$. The restriction of $\widetilde{\varphi}$ on $W^*_p(G,A,\alpha_1)$ is denoted by $\varphi\rtimes \mathrm{id}$. Then $$\varphi\rtimes\mathrm{id}:W^*_p(G,A,\alpha_1)\rightarrow W^*_p(G,B,\alpha_2)$$ is a weak* continuous contractive homomorphism.

 If, in addition, $\varphi$ is a $p$-completely isometric isomorphism, one can check that $$\varphi\rtimes\mathrm{id}:W^*_p(G,A,\alpha_1)\rightarrow W^*_p(G,B,\alpha_2)$$ is a weak* continuous isometric isomorphism.
\end{proof}

To end this section, we study  expectations on $W^*_p(G,A,\alpha)$ which will be useful to prove Proposition \ref{22}.
\begin{lem}\label{exp}
Let $(G,A,\alpha)$ be a weak* $L^p$-operator algebra dynamical system, where $G$ is a countable discrete group, $A\subset \mathcal{B}(L^p(X,\mu))$ is a weak* closed subalgebra, $\alpha:G\rightarrow \mathrm{Aut}(A)$ is a continuous homomorphism. Then for any $t\in G$, there is a weak* continuous contractive linear map $E_t:W^*_p(G,A,\alpha)\rightarrow A$ such that if $f=\sum_{s\in G} a_s\delta_s\in C_c(G,A,\alpha)$, then $E_t(f)=a_t$.
\end{lem}
\begin{proof}
For any $f=\sum_{s\in G} a_s\delta_s\in C_c(G,A,\alpha)$, we define $E_t:C_c(G,A,\alpha)\rightarrow A$ by $E_t(f)=a_t$. Then for any $f\in W^*_p(G,A,\alpha)$, there exists a sequence $f_n\in C_c(G,A,\alpha)$ such that $f_n\xrightarrow{\mathrm{weak}*}f$. We define $E_t:W^*_p(G,A,\alpha)\rightarrow A$ by $$E_t(f)=\lim_{\mathrm{weak}*,n} E_t(f_n),$$ where the $\lim_{\mathrm{weak}*,n}$ is the weak*-limit. Firstly, we shall prove the weak* limit of $E_t(f_n)$ does exist.
By \cite[Exercise 1.35]{Douglas}, it suffices to show that $\la E_t(f_n),x\ra$ is a Cauchy sequence for all $x\in E'\widehat{\otimes}E$. Assume that $x=\sum_{i=1}^\infty \xi_i\otimes\eta_i$, we let $y=\sum_{i=1}^\infty(\delta_e\otimes\xi)\otimes(\delta_{t^{-1}}\otimes\eta_i)\in (l^q(G)\otimes_q E')\widehat{\otimes}(l^p(G)\otimes_p E)$. Then $\la f_n,y\ra$ is a Cauchy sequence. Assume that $f_n=\sum_{s\in G}a_s^{(n)}\delta_s$, a direct computation shows that
$$\begin{aligned} (\sum_{s\in G}a_s^{(n)}\delta_s)(\delta_{t^{-1}}\otimes\eta_i)(r)
 &=\sum_{s\in G}\alpha_{r^{-1}}(a_s^{(n)}(\delta_{t^{-1}}\otimes\eta_i)s^{-1}r)\\
 &=\sum_{s\in G} \alpha_{r^{-1}}(a_s^{(n)}) (\delta_{t^{-1}}\otimes \eta_i)(s^{-1}r)  \ra\\
 &=\alpha_{r^{-1}}(a_{rt}^{(n)})\eta_i,\end{aligned}$$

and
$$\begin{aligned} \la f_n,y\ra
 &=\sum_{i=1}^\infty\la \delta_e\otimes\xi,(\sum_{s\in G}a_s^{(n)}\delta_s)(\delta_{t^{-1}}\otimes\eta_i)\ra\\
 &=\sum_{i=1}^\infty \la \delta_e\otimes\xi, \delta_r\otimes \alpha_{r^{-1}}(a_{rt}^{(n)})\eta_i\ra \\
 &=\sum_{i=1}^\infty\la \xi_i,a_t^{(n)}\eta_i\ra\\
 &=\la E_t(f_n),x\ra.\end{aligned}$$
This shows that  weak* limit of $\{E_t(f_n)\}$ exists.
Then we will prove this Lemma by showing the following three Claims.

{\it Claim 1.} $E_t$ is well defined, that is, $E_t$ does not depend on choice of $f_n$.

Let $f_n=\sum_{s\in G}a_s^{(n)}\delta_s\in C_c(G,A,\alpha)$.
For any $\eta\in L^p(X,\mu)$, a direct computation shows that
$$\begin{aligned} (\sum_{s\in G} a_s^{(n)}\delta_s)(\delta_t\otimes \eta)(r)
 &=\sum_{s\in G}\alpha_{r^{-1}}(a_s^{(n)})\lambda_p^{L^p(X,\mu)}(s)(\delta_t\otimes \eta)(r)\\
 &=\sum_{s\in G} \alpha_{r^{-1}}(a_s^{(n)}) (\delta_t\otimes \eta)(s^{-1}r)  \ra\\
 &=\alpha_{r^{-1}}(a_{rt^{-1}}^{(n)})\eta.\end{aligned}$$
It follows that $$(\sum_{s\in G} a_s^{(n)}\delta_s)(\delta_t\otimes \eta)=\sum_{r\in G}\delta_r\otimes \alpha_{r^{-1}}(a_{rt^{-1}}^{(n)})\eta.$$
Let $x=(\delta_r\otimes\xi)\otimes (\delta_t\otimes \eta)\in (l^q(G)\otimes_p L^q(X,\mu)\widehat{\otimes}(l^p(G)\otimes_p L^p(X,\mu)$, then we have
$$\begin{aligned} \la f_n,x\ra
 &=\la \delta_r\otimes \xi,(\sum_{s\in G}a_s^{(n)}\delta_s)(\delta_t\otimes\eta)\\
 &=\la \delta_r\otimes\xi,\sum_{s\in G}\delta_s\otimes \alpha_{s^{-1}}(a^{(n)}_{st^{-1}})\eta   \ra\\
 &=\la \xi,\alpha_{r^{-1}}(a^{(n)}_{rt^{-1}})\eta \ra. \end{aligned}$$

We choose another $g_n\in C_c(G,A,\alpha)$ such that $g_n\xrightarrow{\mathrm{weak}*}f$,  and let $g_n=\sum_{s\in G}b_s^{(n)}\delta_s$. Then we have
$$\lim\la \xi,\alpha_{r^{-1}}(a^{(n)}_{rt^{-1}})\eta \ra=\lim\la \xi, \alpha_{r^{-1}}(b^{(n)}_{rt^{-1}})\eta \ra.$$
Therefore, $\lim a^{(n)}_{rt^{-1}}=\lim b^{(n)}_{rt^{-1}}$ for all $r,t\in G$. This proves Claim 1.

{\it Claim 2.} $E_t$ is weak* continuous.

Let $f_n,f\in W^*_p(G,A,\alpha)$ such that $f=\lim_{\mathrm{weak}*,n} f_n$, and let $f_{n,m}\in C_c(G,A,\alpha)$ such that $f_n=\lim_{\mathrm{weak}*,m}f_{n,m}$. Then we have $f=\lim_{\mathrm{weak}*,m,n}f_{n,m}$. An easy computation shows that
$$\begin{aligned} E_t(f)
 &=\lim_{\mathrm{weak}*,m,n}E_t(f_{n,m})\\
 &=\lim_{\mathrm{weak}*,n}E_t(f_n). \end{aligned}$$
This proves Claim 2.

{\it Claim 3.} $E_t$ is contractive.

For each $f=\sum_{s\in G}a_s\delta_s \in C_c(G,A,\alpha)$ and $\xi\in l^q(G),\eta\in l^p(G)$, a direct computation shows that
$$\begin{aligned} \la \xi,E_t(f)\eta\ra
 &=\la \xi,a_t\eta\ra \\
 &=\la \delta_e\otimes\xi,\sum_{r\in G}\delta_r\otimes\alpha_{r^{-1}}(a_{rt})\eta\ra\\
 &=\la \delta_e\otimes\xi, f(\delta_t\otimes \eta)\ra. \end{aligned}$$
Since $E_t$ is weak* continuous, it follows that $$ \la \xi,E_t(f)\eta\ra=\la \delta_e\otimes\xi, f(\delta_t\otimes \eta)\ra$$ for all $f\in W^*_p(G,A,\alpha)$. Then we have $$\begin{aligned} \|E_t(f)\|
 &=\sup_{\|\xi\|_q\leq 1,\|\eta\|_p\leq 1}|\la \xi,E_t(f)\eta\ra| \\
 &=\sup_{\|\xi\|_q\leq 1,\|\eta\|_p\leq 1}|\la \delta_e\otimes \xi,f(\delta_t\otimes\eta)\ra|\\
 &\leq \|f\|. \end{aligned}$$
This proves Claim 3.\end{proof}

\begin{lem}\label{inj}
Let $E_t:W^*_p(G,A,\alpha)\rightarrow A$ be as in Lemma \ref{exp}. If $f\in W^*_p(G,A,\alpha)$ and $E_t(f)=0$ for all $t\in G$, then $f=0$.
\end{lem}
\begin{proof}
Let $\nu$ the counting measure on $G$. For each $t,s\in G$, let $V_t:L^p(X,\mu)\rightarrow L^p(G\times X,\nu\times\mu)$ given by $$V_t(\xi)=\delta_t\otimes\xi$$ for all $\xi\in L^p(X,\mu)$, and let $W_s:L^p(G\times X,\nu\times\mu)\rightarrow L^p(X,\mu)$ given by $$W_s(\eta)=\eta(s)$$ for all $\eta\in L^p(G\times X,\nu\times\mu)$.
%Then we have $$W_tV_t(\xi)=(V_t\xi)(t)=\xi$$ and $$V_tW_t(\eta)=\delta_t\otimes W_t(\eta)=\delta_t\otimes \eta(t).$$
We will prove this Lemma by showing the following two Claims.

{\it Claim 1.} If $T\in \mathcal{B}(L^p(G\times X,\nu\times\mu)$ satisfies $W_sTV_t=0$ for all $s,t\in G$, then $T=0$.

For any $s,t\in G$ and $\xi\in L^p(X,\mu)$, we have $$W_sTV_t(\xi)=(TV_t\xi)(s)=T(\delta_t\otimes\xi)(s).$$
Therefore, if $W_sTV_t=0$ for all $s,t\in G$, then $T(\delta_t\otimes\xi)(s)=0$ which implies $T=0$.
This will prove Claim 1.

{\it Claim 2.} $W_s(f)V_t=\alpha_{s^{-1}}(E_{st^{-1}}(f))$ for all $f\in W^*_p(G,A,\alpha)$.

For any $f=\sum_{r\in G}a_r\delta_r\in C_c(G,A,\alpha)$ and $\xi\in L^p(X,\mu)$, a direct computation shows that
$$\begin{aligned} W_s(\sum_{r\in G}a_r\delta_r)V_t(\xi)
 &=(\sum_{r\in G}a_r\delta_r)V_t(\xi)(s) \\
 &=\sum_{r\in G}\alpha_{s^{-1}}(a_r)V_t(\xi)(r^{-1}s)\\
 &=\sum_{r\in G}\alpha_{s^{-1}}(a_r)(\delta_t\otimes \xi)(r^{-1}s)\\
 &=\alpha_{s^{-1}}(a_{st^{-1}})\xi. \end{aligned}$$
This shows that $$W_s(\sum_{r\in G}a_r\delta_r)V_t=\alpha_{s^{-1}}(E_{st^{-1}}(f))$$ for all $f\in C_c(G,A,\alpha)$.

For each $f\in W^*_p(G,A,\alpha)$, there exists a sequence $f_n\in C_c(G,A,\alpha)$ such that $f_n\xrightarrow{\mathrm{weak}*}f$.
For each $x=\sum_{i=1}^\infty \zeta_i\otimes \theta_i\in L^q(X,\mu)\widehat{\otimes}L^p(X,\mu)$, then we have
$$\begin{aligned} \la W_s(f_n)V_t,x\ra
 &=\sum_{i=1}^\infty \la \zeta_i, W_s(f_n)V_t(\theta_i)\ra \\
 &=\sum_{i=1}^\infty \la W_s^*\zeta_i,f_nV_t(\theta_i)\ra\\
 &\rightarrow \sum_{i=1}^\infty \la W_s^*\zeta_i,fV_t(\theta_i)\ra\\
 &=\sum_{i=1}^\infty \la \zeta_i, W_sfV_t(\theta_i)\ra, \end{aligned}$$
where $W_s^*$ is the dual operator of $W_s$. This shows that $$W_s(f_n)V_t\xrightarrow{\mathrm{weak}*} W_sfV_t.$$
Since $E_{st^{-1}}$ is weak* continuous and $\alpha_{s^{-1}}$ is a weak* continuous action, it follows that $$W_s(f)V_t=\alpha_{s^{-1}}(E_{st^{-1}}(f))$$ for all $f\in W^*_p(G,A,\alpha)$. This proves Claim 2.
\end{proof}

\section{Dual action}\label{dual}
If $G$ is abelian, then we let $\hat{G}$ denote the dual group of $G$. For each $\gamma\in \hat{G}$, we denote by $V_\gamma$ the invertible isometric operator on $L^p(G)$ given by $$(V_\gamma f)(s)=\overline{\gamma(s)}f(s)$$ for all $f\in L^p(G)$ and $s\in G$.

\begin{lem}\label{dual1}
Let $\pi_0:A\rightarrow \mathcal{B}(E)$ be the canonical representation, and let $(\pi,\lambda_p^E)$ be its associated regular representation.
For each $\gamma\in \hat{G}$ and $s\in G$,
we have the following:
\begin{enumerate}
\item [(i)] $\left(V_\gamma\otimes I_E\right)\left(\lambda_p^E(s)\right)\left( V_{\gamma^{-1}}\otimes I_E\right)=\overline{\gamma(s)}\lambda_p^E(s)$;
\item [(ii)] $(V_\gamma\otimes I_E)\pi(a)(V_{\gamma^{-1}}\otimes I_E)=\pi(a).$
\end{enumerate}
\end{lem}
\begin{proof}
(i)
Recall that $\lambda_p:G\rightarrow \mathcal{B}(L^p(G))$ is the left regular representation.
To prove (i), it suffices to show that $V_\gamma\lambda_p(s)V_{\gamma^{-1}}=\overline{\gamma(s)}\lambda_p(s)$ for all $s\in G$ and $\gamma\in\hat{G}.$
For each $f\in C_c(G)$, we have that
$$\begin{aligned} \left(V_\gamma\lambda_p(s)V_{\gamma^{-1}}f\right)(t) &=
 \overline{\gamma(t)}\left(\lambda_p(s)V_{\gamma^{-1}}f\right)(t)
  \\ &=\overline{\gamma(t)}\left(V_{\gamma^{-1}}f\right)(s^{-1}t)
  \\ &=\overline{\gamma(t)}\cdot \overline{\gamma^{-1}(s^{-1}t)}f(s^{-1}t)
   \\ &=\overline{\gamma(s)}\left(\lambda_p(s)f\right)(t).\end{aligned}$$
This implies that $V_\gamma\lambda_p(s)V_{\gamma^{-1}}=\overline{\gamma(s)}\lambda_p(s)$.

(ii) If $\xi\in L^p(G,E)$, then $\left((V_\gamma\otimes I_E)\xi\right)(s)=\overline{\gamma(s)}\xi(s)$.
Then we get
$$\begin{aligned} \left((V_\gamma\otimes I_E)\pi(a)\xi\right)(s) &=
\overline{\gamma(s)}\left(\pi(a)\xi\right)(s)
  \\ &=\overline{\gamma(s)}\pi_0(\alpha_{s^{-1}}(s))\xi(s)
  \\ &=\pi_0(\alpha_{s^{-1}}(s)) (V_\gamma\otimes I_E)\xi(s)
   \\ &=\left(\pi(a)(V_\gamma\otimes I_E)\xi\right)(s).\end{aligned}$$
Hence $(V_\gamma\otimes I_E)\pi(a)(V_{\gamma^{-1}}\otimes I_E)=\pi(a).$
\end{proof}

To proceed, we recall the compact open topology on $\hat{G}$. A neighbourhood basis of $\alpha_0 \in \hat{G}$ is formed by the collection of sets
$$V(\alpha_0, K, \epsilon) = \{ \alpha \in \hat{G} : |\alpha(x) - \alpha_0(x)| < \epsilon \text{ for all } x \in K \},$$
where $\epsilon>0$ and $K$ is any compact subset of $G$. Then $\hat{G}$ is a topological group since $V(\alpha_0,K,\epsilon) = V(\alpha_0^{-1}, K, \epsilon)$ and
$V(\alpha_0, K, \epsilon)V(\beta_0, K, \epsilon) \subset V(\alpha_0\beta_0, K, 2\epsilon)$.

\begin{prop}
For each $\gamma\in\hat{G}$, we define $\hat{\alpha}_\gamma(f)=(V_\gamma\otimes I_E)f(V_{\gamma{-1}}\otimes I_E)$ for all $f\in W^*_p(G,A,\alpha)$. Then  $\hat{\alpha}:\hat{G}\rightarrow \mathrm{Aut}(W^*_p(G,A,\alpha)$ is a continuous homomorphism.  The action $\hat{\alpha}$ is called the dual action.
\end{prop}
\begin{proof}
It follows from Lemma \ref{dual1} that $\hat{\alpha}$ is an isometric action of $\hat{G}$ on $W^*_p(G,A,\alpha)$ and $\hat{\alpha}$ is a homomorphism.

For each $f\in W^*_p(G,A,\alpha)$,
to prove $\hat{\alpha}:\hat{G}\rightarrow \mathrm{Aut}(W^*_p(G,A,\alpha)$ is continuous, it suffices to show that $$\hat{\alpha}(f):\hat{G}\rightarrow W^*_p(G,A,\alpha); \gamma\mapsto\hat{\alpha}_\gamma(f)$$ is continuous

{\it Claim 1.} Let $\gamma\in \hat{G}$. If a net $\{f_n\}$ converges to $f$ in weak* topology, then $\{\hat{\alpha}_\gamma(f_n)\}$ converges $\hat{\alpha}_\gamma(f)$ in weak* topology.

 Let $F=L^p(G)\otimes_p E$. For each $\omega=\sum_{i=1}^\infty\xi_i\otimes\eta_i\in \mathcal{N}(F)$ and $\gamma\in\hat{G}$,  we have
$$\begin{aligned} \la \hat{\alpha}_\gamma(f_n),\omega\ra &=
\sum_{i=1}^\infty\la\xi_i,(V_\gamma\otimes I_E)(f_n)( V_{\gamma^{-1}}\otimes I_E)\eta_i\ra
  \\ &=\sum_{i=1}^\infty\la (V_\gamma\otimes I_E)'\xi_i,(f_n)(V_{\gamma^{-1}}\otimes I_E)\eta_i\ra
  \\ &=\la f_n,\sum_{i=1}^\infty ((V_\gamma\otimes I_E)'\xi_i)\otimes ((V_{\gamma^{-1}}\otimes I_E)\eta_i\ra
   \\ &\rightarrow \la f,\sum_{i=1}^\infty ((V_\gamma\otimes I_E)'\xi_i)\otimes ((V_{\gamma^{-1}}\otimes I_E)\eta_i\ra
   \\&=\sum_{i=1}^\infty\la (V_\gamma\otimes I_E)'\xi_i,(f)(V_{\gamma^{-1}}\otimes I_E)\eta_i\ra
   \\&=\sum_{i=1}^\infty\la\xi_i,(V_\gamma\otimes I_E)(f)(V_{\gamma^{-1}}\otimes I_E)\eta_i\ra
   \\&=\la \hat{\alpha}_\gamma(f),\omega\ra,\end{aligned}$$
where $(V_\gamma\otimes I_E)'$ is dual operator of $V_\gamma\otimes I_E$.
This prove Claim 1.

Let $N(\hat{\alpha}_{\gamma_0}(f))=\cap_{i=1}^n\{g\in W^*_p(G,A,\alpha):|\la \hat{\alpha}_{\gamma_0}(f)-g,\omega_i\ra|<1,\omega_i\in \mathcal{N}(F)\}$
be a neighborhood of $\hat{\alpha}_{\gamma_0}(f)$.
By Claim 1, there exists a $h\in C_c(G,A,\alpha)$ such that $$|\la \hat{\alpha}_{\gamma_0}(\pi\rtimes \lambda_p^E(h))-\hat{\alpha}_{\gamma_0}(f),\omega_i\ra|<\frac{1}{3}.$$ We may assume that $h\neq 0$ and $\la\pi\rtimes \lambda_p^E(h),\omega_i\ra\neq 0$ for all $i\in\{1,\cdots,n\}$. Let $K=\mathrm{supp} ~h$, then $K$ is a compact subset of $G$.
Let $\delta=\frac{1}{3}\min\{|\la\pi\rtimes \lambda_p^E(h),\omega_i\ra|^{-1},i=1,\cdots,n\},$ and let $N(\gamma_0)=\{\gamma\in \hat{G}:|\gamma(t)-\gamma_0(t)|<\delta,t\in K\}$. We will finish the proof by showing the following Claim 2.

{\it Claim 2.} $\hat{\alpha}(f)(N(\gamma_0))\subset N(\hat{\alpha}_{\gamma_0}(f))$.

For each $\omega_i\in \mathcal{N}(F)$ and $\gamma\in \hat{\alpha}(N(\gamma_0))$, a direct computation shows that
$$\begin{aligned} &|\la \hat{\alpha}_{\gamma_0}(\pi\rtimes_p^E(h))-\hat{\alpha}_\gamma(\pi\rtimes_p^E(h)),\omega_i\ra|\\
 &=|\la(I_E\otimes V_{\gamma_0})(\pi\rtimes_p^E(h))(V_{{\gamma_0}^{-1}}\otimes I_E)-(I_E\otimes V_{\gamma})(\pi\rtimes_p^E(h))(V_{{\gamma}^{-1}}\otimes I_E),\omega_i\ra|
  \\ &=|\la \int_G(\overline{\gamma_0(t)}-\overline{\gamma(t)})\pi(h(t))\lambda_p^E(t)d\mu(t),\omega_i\ra|<\frac{1}{3}\end{aligned}$$
and
$$\begin{aligned}& |\la \hat{\alpha}_{\gamma_0}(f)- \hat{\alpha}_{\gamma}(f),\omega_i\ra| \\
&\leq|\la \hat{\alpha}_{\gamma_0}(f)- \hat{\alpha}_{\gamma_0}(\pi\rtimes_p^E(h)),\omega_i\ra|+|\la \hat{\alpha}_{\gamma_0}(\pi\rtimes_p^E(h))- \hat{\alpha}_{\gamma}(\pi\rtimes_p^E(h)),\omega_i\ra|\\
&+|\la \hat{\alpha}_{\gamma}(\pi\rtimes_p^E(h))- \hat{\alpha}_{\gamma}(f),\omega_i\ra|
  \\ &<\frac{1}{3}+\frac{1}{3}+\frac{1}{3}=1.\end{aligned}$$
This proves Claim 2.
\end{proof}

\section{The $C^*$-core of $L^p$-operator algebras}
The $C^*$-core of an $L^p$-operator algebra was introduced by Y. Choi,  E. Gardella and H. Thiel \cite{Choi and}, which is an invariant under isometric isomorphisms of $L^p$-operator algebras.

\begin{defn}[{\cite[Definition 2.6.1]{Palmer}}]
Let $A$ be a unital Banach algebra in which $\|1\|=1$. The {\it numerical range} $W(a)$ of an element $a$ in $A$ is the set of all numbers $w(a)\in \mathbb{C}$ with $w$ being a continuous linear functional on $A$ with $\|w\|=w(1)=1$.
\end{defn}

\begin{defn}[{\cite[Definition 5.5]{Lp AF}}]
Let $A$ be a unital Banach algebra in which $\|1\|=1$.
An element $a\in A$ is said to be {\it hermitian} if $W(a)\subset \mathbb{R}$.
\end{defn}

\begin{defn}[{\cite[Definition 2.10]{Choi and}}]
Let $p\in [1,\infty)$, and let $A$ be a unital $L^p$-operator algebra. Denote by $A_h$ the set of hermitian elements in $A$. The algebra $\mathrm{core}(A):={A_h}+\mathrm{i}{A_h}$ is called the $C^*$-{\it core} of $A$.
\end{defn}

By \cite[Theorem 2.9]{Choi and}, the $C^*$-core of $A$ is the largest unital $C^*$-subalgebra of $A$, and it is commutative when $p\in [1,\infty)\setminus\{2\}$.
By \cite[Proposition 2.13]{Choi and}, if $\varphi:A\rightarrow B$ is an isometric isomorphism between two $L^p$-operator algebras $A$ and $B$, then $\varphi: \mathrm{core}(A)\rightarrow \mathrm{core}(B)$ is a $*$-isomorphism. This implies that the $C^*$-core is invariant under isometric isomorphism between $L^p$-operator algebras.

\begin{prop}\label{core1}
Let $p\in (1,\infty)\setminus\{2\}$. Let $G$ be a countable discrete group, $A\subset \mathcal{B}(E)$ be a weak* closed $L^p$-operator algebra, and $\alpha$ be a weak* continuous isometric action of $G$ on $A$. Then the canonical embedding $A\subset W^*_p(G,A,\alpha)$ identifies the $C^*$-core of $W^*_p(G,A,\alpha)$ with that of $A$, that is,
$\mathrm{core}(W^*_p(G,A,\alpha))=\mathrm{core}(A)$.
\end{prop}
\begin{proof}
Let $\pi_0:A\rightarrow \mathcal{B}(E)$ be the canonical representation, and let $(\pi,\lambda_p^E)$ be its associated regular representation.
For each $s,t\in G$ and $\xi\in E$,
we have $$\lambda_p^E(s)(\delta_s\otimes\xi)=\delta_{st}\otimes\xi$$ and $$\pi(a)(\delta_t\otimes\xi)=\delta_t\otimes \pi_0(\alpha_{t^{-1}}(a))(\xi).$$
We define $F:\mathcal{B}(l^p(G)\otimes_p E)\rightarrow \mathcal{B}(E)$ by $$F(T)(\xi)=(T(\delta_e\otimes\xi))(e)$$ for all $T\in \mathcal{B}(l^p(G)\otimes_p E))$ and $\xi\in E$, where $e$ is the unit of $G$.

{\it Claim 1.} $F$ is weak* continuous.

Let $\{T_\alpha\}$ be a net in $\mathcal{B}(l^p(G)\otimes_p E)$ such that $\{T_\alpha\}$ converges to $T$ in weak* topology.
For each $x=\sum_{i=1}^\infty \xi_i\otimes\eta_i\in E'\widehat{\otimes}E$, we have
$$\begin{aligned} \la F(T_\alpha),x\ra
&=\sum_{i=1}^\infty \la \xi_i, F(T_\alpha)\eta_i\ra\\
&=\sum_{i=1}^\infty \la \xi_i, T_\alpha(\delta_e\otimes\eta_i)(e)\ra\\
&=\sum_{i=1}^\infty \la \delta_e\otimes \xi_i,T_\alpha(\delta_e\otimes\eta_i) \ra \\
&=\sum_{i=1}^\infty \la T_\alpha, \sum_{i=1}^\infty (\delta_e\otimes\xi_i)\otimes (\delta_e\otimes\eta_i) \ra \\
&\rightarrow \sum_{i=1}^\infty \la T, \sum_{i=1}^\infty (\delta_e\otimes\xi_i)\otimes (\delta_e\otimes\eta_i) \ra \\
&=\sum_{i=1}^\infty \la \delta_e\otimes \xi_i,T(\delta_e\otimes\eta_i) \ra \\
&=\sum_{i=1}^\infty \la \xi_i, T(\delta_e\otimes\eta_i)(e)\ra\\
&=\sum_{i=1}^\infty \la \xi_i, F(T)\eta_i\ra\\
&=\la F(T),x\ra.
\end{aligned}$$
This proves Claim 1.

{\it Claim 2.} We have $F\circ(\pi\rtimes\lambda_p^E)=\pi_0\circ F$, that is, the following diagram commutates:
$$\begin{CD}
  W^*_p(G,A,\alpha) @ >\pi\rtimes\lambda_p^E>>\mathcal{B}(l^p\otimes_p E) \\
     @VV E_e V @VV F V\\
    A @>\pi_0>> A.
\end{CD}$$
Since all the maps involved are weak* continuous, it suffices to prove this by showing $$F\circ(\pi\rtimes\lambda_p^E)(f)=\pi_0\circ F(f)$$ for all $f\in C_c(G,A,\alpha)$. Let $f=\sum_{s\in G}a_s\delta_s$ and $\xi\in E$, a direct computation shows that

$$\pi_0\circ E_e(\sum_{s\in G}a_s\delta_s)=a_e,$$

   and $$\begin{aligned} F\circ(\pi\rtimes\lambda_p^E)(f)\xi &=
(\pi\rtimes\lambda_p^E)(f)(\delta_e\otimes\xi)(e)
  \\ &=\sum_{s\in G}\pi(a_s)\lambda_p^E(s)(\delta_e\otimes\xi)(e)
   \\ &=\sum_{s\in G}(\delta_s\otimes \pi_0(\alpha_{s^{-1}}(a_s))\xi)(e)
   \\&=a_e\xi.\end{aligned}$$
This proves Claim 2.

{\it Claim 3.}
Let $E=L^p(X,\mu)$ be a $\sigma$-finite separable $L^p$-space. If $\eta:G\rightarrow L^\infty(X,\mu)$ is a bounded function with associated multiplication operator $M_\eta\in \mathcal{B}(l^p(G)\otimes_p E)$, then $F(M_\eta\lambda_p^E(t))=0$ when $t\neq e$.

For each $\xi\in E$ and $t\in G, t\neq e$, we have
$$\begin{aligned} F(M_\eta\lambda_p^E(t))(\xi) &=
(M_\eta\lambda_p^E(t))(\delta_e\otimes\xi)(e)
  \\ &=\eta(e)((\delta_t\otimes\xi)(e))
   \\ &=0.\end{aligned}$$
This proves Claim 3.

{\it Claim 4.} Let $f\in \mathrm{core}(W^*_p(G,A,\alpha))$. Then $E_e(\pi\rtimes\lambda_p^E(f)\lambda_p^E(t))=0$ when $t\neq e$.

Note that $\pi\rtimes\lambda_p^E(f)\in \mathrm{core}(\mathcal{B}(l^p(G)\otimes_p E))$, it follows from \cite[Example 2.11]{Choi and} that there exists a bounded function $\eta:G\rightarrow E$ such that $M_\eta=\pi\rtimes\lambda_p^E(f)$. For each $t\in G, t\neq e$,
by Claim 2 and Claim 3, we have $$\begin{aligned} \pi_0(E_e(\pi\rtimes\lambda_p^E(f)\lambda_p^E(t))) &=
F(\pi\rtimes\lambda_p^E(f)\lambda_p^E(t))
  \\ &=F(M_\eta\lambda_p^E(t))
   \\ &=0.\end{aligned}$$
This proves Claim 4.

Now we shall prove this proposition. Obviously, we have $\mathrm{core}(A)\subset \mathrm{core}(W^*_p(G,A,\alpha))$. If $f\in \mathrm{core}(W^*_p(G,A,\alpha))$, to finish the proof of this proposition, it suffices to show that $E_e(f)=f$. By \cite[Proposition 2.16]{Choi and}, we have $E_e(f)\in \mathrm{core}(A)$.Thus for each $t\in G\setminus\{e\}$, we have $$E_e(f\lambda_p^E(t))=0=E_e(E_e(f)\lambda_p^E(t)).$$
For $t=e$, we have $$E_e(f\lambda_p^E(e)=E_e(f)=E_e(E_e(f)\lambda_p^E(e)).$$
This shows that $$E_t(E_e(f)-f)=E_e((E_e(f)-f)\lambda_p^E(t^{-1}))=0$$ for each $t\in G$. Therefore, $E_e(f)=f$.
\end{proof}

%\begin{prop}
%Let $A\subset \mathcal{B}(L^p(X,\mu))$ and $B\subset\mathcal{B}(L^p(Y,\nu))$ be two weak* closed $L^p$-operator algebras. Then we have $\mathrm{core}(A\bar{\otimes}B)=\mathrm{core}(A)\bar{\otimes}\mathrm{core}(B)$.
%\end{prop}

\section{The Gelfand theory of $p$-pseudomeasure algebras}

In this section, we study the Gelfand theory of $p$-pseudomeasure algebra $PM_p(G)$ when $G$ is a locally compact abelian group. The reader will see that then Takesaki duality heavily depend on Gelfand theory. All the new phenomena of Takesaiki duality for weak* closed $L^p$-operator crossed products originally come from the Gelfand theory of $p$-pseudomeasure algebras.
Firstly, we look into the $C^*$-core of $PM_p(G)$.

\begin{lem}\label{core2}
Let $p\in (1,\infty)\setminus\{2\},$ and let $G$ be a locally compact group. If $PM_p(G)=CV_p(G)$, then $\mathrm{core}(PM_p(G))=\mathbb{C}$.
\end{lem}
\begin{proof}
The algebras $PM_p(G)$ and $CV_p(G)$ are always unital (see \cite{Gardella and Thiel iso}). For any $f\in (PM_p(G))_h$, then $f$ is a multiplication operator on $L^p(G)$. Since $f\in PM_p(G)=CV_p(G)$, it follows that $$f\rho_p(t)=\rho_p(t)f$$ for all $t\in G$, where $\rho_p$ is the right regular representation of $G$. This implies that $$f(s)=f(st)$$ for all $s,t\in G$. Then $f$ is a real-valued constant function. Therefore, $\mathrm{core}(PM_p(G))=\mathbb{C}$.
\end{proof}

Secondly,
we introduce $p$-incompressible $L^p$-operator algebras. This property is invariant under isomorphisms of $L^p$-operator algebras. The definition is inspired by N. C. Phillips and B. Johnson' work on incompressible Banach algebras \cite{incompressible}.

\begin{defn}
Let $A$ be an $L^p$-operator algebra.
We say that $A$ is {\it $p$-incompressible} if whenever  $E$ is an $L^p$-space and $\varphi:A\rightarrow \mathcal{B}(E)$ is a bounded injective homomorphism, then $\varphi$ is bounded below.
\end{defn}

\begin{rmk}\label{c}
\begin{enumerate}
\item[(i)] Let $A$ and $B$ be two $L^p$-operator algebras. Assume that $A$ and $B$ are isomorphic, then $B$ is $p$-incompressible if and only if $A$ is $p$-incompressible.
\item[(ii)] The $M_n^p$ and $C^*$-algebras are examples of $p$-incompressible $L^p$-operator algebras (see \cite{incompressible}).
\end{enumerate}
\end{rmk}

\begin{lem}[{\cite[Exercise 5 on page 93]{Conway}}]\label{below}
Let $X$ and $Y$ be Banach spaces and let $T\in \mathcal{B}(X,Y)$. Then there is a constant $c>0$ such that $\|Tx\||\geq c\|x\|$ for all $x\in X$ if and only if $\ker T=\{0\}$ and $\mathrm{ran} T$ is closed.
\end{lem}

The algebra $L^\infty(G)$ is a dual space of $L^1(G)$, then there is a natural weak* topology of $L^\infty(G)$ viewed as the dual space of $L^1(G)$. Also, there is a natural multiplicative representation $M:L^\infty(G)\rightarrow \mathcal{B}(L^p(G)); f\mapsto M_f$, where $M_f$ is the multiplicative operator on $L^p(G)$. So there is a weak* topology on $L^\infty(G)$  viewed as the operators on $L^p(G)$, and we call this topology  {\it operator weak* topology}. In the following lemma, we will show that this two weak* topology coincide.
\begin{lem}
Let $p>1$ and let $G$ be a locally compact group. The natural weak* topology on $L^\infty(G)$ coincides with operator weak* topology.
\end{lem}
\begin{proof}
Let $\{f_\alpha\}$ be a net in $L^\infty(G)$ converges to $f\in L^\infty(G)$ in weak* topology. By uniform boundedness principle we may assume
that $\{f_\alpha\}$ is uniformly bounded. For each $\xi\in L^q(G)$ and $\eta\in L^p(G)$, by H$\ddot{\mathrm{o}}$lder inequality, we have $\xi\eta\in L^1(G)$. Then $$\la \xi,M_{f_\alpha}\eta\ra=\la f_\alpha, \xi\eta\ra\rightarrow \la f,\xi\eta\ra=\la \xi,M_f\eta\ra.$$
Since the operator weak* topology coincides with weak operator topology on bounded sets, it follows that the net $\{f_\alpha\}$ converges to $f$ in operator weak* topology.

On the other hand, let $\{f_\alpha\}$ be a net in $L^\infty(G)$ converges to $f\in L^\infty(G)$ in operator weak* topology. For each $g\in L^1(G)$, we let $g^{\frac{1}{p}}(t)=g(t)|g(t)|^{\frac{1}{p}-1}\lambda(t)$ and $g^{\frac{1}{q}}(t)=g(t)|g(t)|^{\frac{1}{q}-1}$, where $g(t)=\lambda(t)|g(t)|$. Then we have $g^{\frac{1}{p}}\in L^p(G)$ and $g^{\frac{1}{q}}\in L^q(G)$. Therefore, one can check that $$\la f_\alpha,g\ra=\la f_\alpha, g^{\frac{1}{p}}g^{\frac{1}{q}}\ra=\la g^{\frac{1}{q}}, M_{f_\alpha} g^{\frac{1}{p}}\ra\rightarrow \la g^{\frac{1}{q}}, M_{f} g^{\frac{1}{p}}\ra=\la f,g\ra.$$
Hence the net $\{f_\alpha\}$ converges to $f$ in weak* topology.
\end{proof}

\begin{thm}\label{G}
Let $p>1$ and Let $G$ be a locally compact abelian group. We denote by $\Gamma_p:PM_p(G)\rightarrow L^\infty(\hat{G})$ the Gelfand transform of $PM_p(G)$. Then:
\begin{enumerate}
\item [(i)] the map $\Gamma_p$ is injective contractive weak* continuous with dense range in weak* topology;
\item [(ii)] the map $\Gamma_p$ is surjective if and only if $G$ is finite or $p=2$;
\item [(iii)] the map $\Gamma_p$ is isometrically isomorphic when $p=2$ or $G$ is trivial;
\item [(iv)] $PM_p(G)$ is isometrically isomorphic to $L^\infty(\hat{G})$ if and only if $p=2$ or $G$ is trivial;
\item [(v)] $PM_p(G)$ is isomorphic to $L^\infty(\hat{G})$ if and only if $G$ is finite or $p=2$.
\end{enumerate}
%(i) $\Gamma_p$ is injective contractive with dense range in weak* topology;
%
%(ii) $\Gamma_p$ is surjective if and only if $G$ is finite or $p=2$;
%
%
%(iii) $\Gamma_p$ is weak* continuous;
%
%(iv) $PM_p(G)$ is isometrically isomorphic to $L^\infty(\hat{G})$ if and only if $p=2$ or $G$ is trivial;
%
%(v) $PM_p(G)$ is isomorphic to $L^\infty(\hat{G})$ if and only if $G$ is finite or $p=2$.
\end{thm}
\begin{proof}
Denote by $\hat{G}$ the dual group of $G$,
by \cite[Theorem 2 on page 13]{Der}, it follows that the Gelfand transform $\Gamma_2:PM_2(G)\rightarrow L^\infty(\hat{G})$ is an isometric isomorphism.

(i) Combing with \cite[Theorem 4 on page 19]{Der} we have that the Gelfand transform $\Gamma_p: PM_p(G)\rightarrow L^\infty(\hat{G})$ is injective and contractive. Obviously, $\Gamma_p$ has dense range in weak* topology.
It follows from \cite[Theorem 7 on page55]{Der} that the inclusion $A_2(G)\rightarrow A_p(G)$ is contractive, and this inclusion is donated by $\iota$.
Let $\iota':(A_p(G))'\rightarrow (A_2(G))'$ be the dual map. Since $\iota$ is bounded, it follows $\iota'$ is weak* continuous. Note that $PM_p(G)=A_p(G)'$, $PM_2(G)=A_2(G)'$, $\Gamma_p=\Gamma_2\circ\iota'$ and $\Gamma_2$ is weak* continuous, then we have that $\Gamma_p$ is weak* continuous.

(ii) If $G$ is finite, then $\Gamma_p$ is isomorphic. If $G$ is infinite and $p\neq 2$, it follows that \cite[Remark 1 on page 54]{Der}, \cite{F} or \cite[Theorem 4.5.2]{Lar} that $\Gamma_p$ is not surjective. Hence $\Gamma_p$ is surjective if and only if $G$ is finite.

(iii) If $p=2$ or $G$ is trivial, obviously, we have that $\Gamma_2$ is an isometric isomorphism.

(iv) If $p=2$ or $G$ is trivial, then $PM_p(G)$ is isometrically isomorphic to $L^\infty(\hat{G})$.
If $p\neq 2$ and $G$ is not trivial, then the $C^*$-core of $PM_p(G)$ is $\mathbb{C}$ (see Lemma \ref{core2}) and the $C^*$-core of $L^\infty(\hat{G})$ is $L^\infty(\hat{G})$. Since the isometric isomorphisms between Banach algebras preserve $C^*$-cores, it follows that $PM_p(G)$ is not isometrically isomorphic to $L^\infty(\hat{G})$.

(v) If $p=2$ or $G$ is finite, then we have that $PM_p(G)$ is isometrically isomorphic to $L^\infty(\hat{G})$. If $p\neq 2$ and $G$ is infinite, it follows that Remark \ref{c} (ii) that $L^\infty(\hat{G})$ is $p$-incompressible.
By (i) (ii) and Lemma \ref{below}, it follows that $PM_p(G)$ is not incompressible. Since the isomorphisms between $L^p$-operator algebras preserve $p$-incompressibility, it follows that $PM_p(G)$ is not isomorphic to $L^\infty(\hat{G})$.
\end{proof}

\section{Proof of Theorem \ref{Thm1}}

In this section, we prove the Theorem \ref{Thm1}. We will construct the homomorphism $\Phi$ as compositions:
$$\begin{CD}
W^*_p\left(\hat{G},F^{p}(G,A,\alpha),\hat{\alpha}\right) @>\Phi_1>> W^*_p\left(G,
W^*_p(\hat{G},A,\beta),\hat{\beta}\otimes \alpha\right)\\
@.   @VV\Phi_2V\\
W^*_p\left(G,l^\infty(G)\bar{\otimes}A,\mathrm{lt}\otimes\mathrm{id}\right) @<\Phi_3<<W^*_p\left(G,l^\infty(G)\bar{\otimes}A,\mathrm{lt}\otimes\alpha\right)\\
@VV\Phi_4V @. \\
\mathcal{B}(l^p(G))\bar{\otimes}A .
\end{CD}$$

The precise definition of group actions and homomorphisms in this diagram will be provided in the subsections of this section.
To prove Theorem \ref{Thm1}, it suffices to prove the following theorem.
\begin{thm}\label{m}
Let $p\in(1,\infty)$. Let $\Phi=\Phi_4\circ\Phi_3\circ\Phi_2\circ\Phi_1$.
\begin{enumerate}
\item [(i)] the maps $\Phi_1$, $\Phi_3$, and $\Phi_4$ are weak* continuous isometrically isomorphic;
\item [(ii)] the map $\Phi_2$ is weak* continuous contractive injective;
\item [(iii)] the map $\Phi_2$ is isometrically isomorphic if and only if $p=2$ or $G$ is trivial;
\item [(iv)]  the map $\Phi_2$ is surjective if and only if $p=2$ or $G$ is finite;
\item [(v)] the map $\Phi$ is equivariant for the double dual action $\hat{\hat{\alpha}}$ of $G$ on $W^*_p(\hat{G},W^*_p(G,A,\alpha),\hat{\alpha})$ and the action $\mathrm{Ad}\rho_p\otimes\alpha$ of $G$ on $\mathcal{B}(l^p(G))\bar{\otimes}A$, where $$\hat{\hat{\alpha}}_{t}(F)(\gamma, s):=\overline{\gamma(t)} F(\gamma, s)$$ for all $t\in G, F\in C_c(\hat{G}\times G,A)\subset W^*_p(\hat{G},W^*_p(G,A,\alpha),\hat{\alpha})$, and
   $$(\mathrm{Ad}\rho_p\otimes\alpha)_s (T\otimes a)=(\mathrm{Ad}\rho_p)_s(T)\otimes \alpha_s(a)$$ for all $T\in \mathcal{B}(l^p(G)), a\in A$.
\end{enumerate}
\end{thm}
The proof of Theorem \ref{m} follows Propositions \ref{1}, \ref{22}, \ref{3}, \ref{4}, \ref{equi}.

\subsection{The weak* continuous isometry $\Phi_1$}

Firstly, we introduce the action of $G$ on $W^*_p(\hat{G},A,\iota)$.
For each $s\in G$, we denoted by $W_s\in \mathcal{B}(L^p(\hat{G})$ be the invertible isometric operator given by $$W_s\xi(\gamma)=\overline{\gamma(s)}\xi(\gamma)$$ for all $\xi\in L^p(\hat{G})$.
Let $\iota$ be the trivial action of $\hat{G}$ on $A$. We denote by $\hat{\iota}$ the dual action of $G$ on $PM_p(\hat{G})$, that is, $$\hat{\iota}_s(f)=W_s(f)W_{s^{-1}}$$ for all $f\in PM_p(\hat{G})$.
One can check that $\hat{\iota}_s(f)(\gamma)=\overline{\gamma(s)}f(\gamma)$ for all $f\in C(\hat{G})$.
Since $\iota$ is the trivial action of $\hat{G}$ on $A$, it follows that $PM_p(\hat{G},A,\iota)=PM_p(\hat{G})\bar{\otimes}A$.
Let $\beta$ be the action of $G$ on $W^*_p(\hat{G},A,\iota)$ given by $$\beta_s=\hat{\iota}_s^{-1}\otimes\alpha.$$
One can check that $$\beta_s(F)(\gamma)=\gamma(s)\alpha_s(F(\gamma))$$ for all $F\in C(\hat{G},A,\iota)$.

\begin{lem}
Let $\alpha$ be a $p$-completely isometric action of $G$ on $A$, and let $\beta$ be given as above. Then the triple $(G,W^*_p(\hat{G},A,\iota),\beta)$ is a weak* closed $L^p$-operator algebra dynamical system.
\end{lem}
\begin{proof}
Let $\mathrm{id}_{PM_p(\hat{G})}$ and $\mathrm{id}_A$ be the identity maps on $PM_p(\hat{G})$ and $A$, respectively.
Note that $\beta_s=(\mathrm{id}_{PM_p(\hat{G})}\otimes \alpha_s)(\hat{\iota}_s^{-1}\otimes\mathrm{id}_A)$, since $\alpha_s$ is $p$-completely isometrically isomorphic, it follows from Lemma \ref{tensor} that $$\|\beta_s(F)\|=\|F\|.$$
This shows that $\beta$ is an isometric action of $G$ on $W^*_p(\hat{G},A,\iota)$.

Since $G$ is discrete, it follows that $$\beta(F):G\rightarrow W^*_p(\hat{G},A,\iota); s\mapsto \beta_s(F)$$ is continuous for all $F\in W^*_p(\hat{G},A,\iota)$.
\end{proof}

\begin{prop}\label{1}
Let $A\subset \mathcal{B}(E)$ be a weak* closed $L^p$-operator algebra, $G$ be a countable discrete locally compact Abelian group, and let $\alpha:G\rightarrow \mathrm{Aut}(A)$ be a weak* continuous isometric action. Then there exists a dual weak* $L^p$-operator algebra dynamical system $(\hat{G},W^*_p(G,A,\alpha),\hat{\alpha})$. Let $\iota$ be the trivial action of $\hat{G}$ on $A$. There exists a dual weak* $L^p$-operator algebra dynamical system $(G,W^*_p(\hat{G},A,\iota),\beta)$ such that $\left(\beta_t(f)\right)(\tau)=\tau(t)\alpha_t\left(f(\tau)\right)$ for all $f\in C_c(\hat{G},A,\iota)$.
Then there exists a weak* continuous isometric isomorphism $\Phi_1: W^*_p(\hat{G},W^*_p(G,A,\alpha),\hat{\alpha})\rightarrow W^*_p(G,W^*_p(\hat{G},A,\iota),\beta)$.
\end{prop}
\begin{proof}
We denote by $\mu,\nu$ be the Haar measure on $G$ and $\hat{G}$, respectively.
The algebra $W^*_p(\hat{G},W^*_p(G,A,\alpha),\hat{\alpha})$ has an isometric representation on $L^p(\hat{G}\times G,E)$.
For any $x\in C_c(\hat{G}\times G,A)$ and $\xi\in L^p(\hat{G}\times G,E)$, then $x(\tau)\in C_c(G,A,\alpha)$ and $\xi(\tau)\in L^p(G,E)$ for all $\tau\in\hat{G}$.
By the definition of integrated form we have
 $$(x\xi)(\tau)=\int_{\hat{G}}\hat{\alpha}_{\tau^{-1}}\left(x(\gamma)\right)\xi(\gamma^{-1}\tau)d\nu(\gamma),$$
and $$\begin{aligned} \left(\hat{\alpha}_{\tau^{-1}}\left(x(\gamma)\right)\xi(\gamma^{-1}\tau)\right)(t)&=
\int_G \alpha_{t^{-1}}\left(\hat{\alpha}_{\tau^{-1}}\left(x(\gamma)\right)(s)\right)\xi(\gamma^{-1}\tau,s^{-1}t)d\mu(s)
  \\ &=\int_G \overline{\gamma^{-1}(s)}\alpha_{t^{-1}}\left(x(\gamma,s)\right)\xi(\gamma^{-1}\tau,s^{-1}t)d\mu(s).\end{aligned}$$
It follows that $$x\xi(\tau,t)=\int_{\hat{G}}\int_G \gamma(s)\alpha_{t^{-1}}\left(x(\gamma,s)\right)\xi(\gamma^{-1}\tau,s^{-1}t)d\mu(s)d\nu(\gamma).$$

Similarly, the algebra $W^*_p(G,W^*_p(\hat{G},A,\iota),\beta)$ has an isometric representation on $L^p(G\times\hat{G},E)$.
For each $y\in C_c(G\times \hat{G},A)$, and $\eta\in L^p(G\times\hat{G},E)$, then $y(t)\in C_c(\hat{G},A,\iota)$ and $\eta(t)\in L^p(\hat{G},E)$ for all $t\in G$. It follows from the integrated form that $$(y\eta)(t)=\int_G \beta_{t^{-1}}\left(y(s)\right)\eta(s^{-1}t)d\mu(s).$$
Also we have
$$\begin{aligned} \left(\beta_{t^{-1}}\left(y(s)\right)\eta(s^{-1}t)\right)(\tau)&=
\int_{\hat{G}} \iota_{\tau^{-1}}\left(\beta_{t^{-1}}(y(s))(\gamma)\right)\eta(s^{-1}t,\gamma^{-1}\tau)d\nu(\gamma)
  \\ &=\int_{\hat{G}} \gamma(t^{-1})\alpha_{t^{-1}}\left(y(s,\gamma)\right)\eta(s^{-1}t,\gamma^{-1}\tau)d\nu(\gamma).\end{aligned}$$
It follows that $$y\eta(t,\tau)=\int_G\int_{\hat{G}}\overline{\gamma(t)}\alpha_{t^{-1}}\left(y(s,\gamma)\right)\eta(s^{-1}t,\gamma^{-1}\tau)d\nu(\gamma)d\mu(s).$$

We define an isomorphism $\Phi_1:C_c(\hat{G}\times G,A)\rightarrow C_c(G\times \hat{G},A)$ given by $$\left(\Phi_1(x)\right)(t,\tau)=\tau(t)x(\tau,t).$$
Furthermore, we define an isometry $W:L^p(\hat{G}\times G,E)\rightarrow L^p(G\times \hat{G},E)$ given by $$W\xi(t,\tau)=\overline{\tau(t)}\xi(\tau,t).$$
Then for each $x\in C_c(\hat{G}\times G,A)$ and $\xi\in L^p(\hat{G}\times G,E)$, we have
$$\begin{aligned} &\left(\Phi_1(x)(W\xi)\right)(t,\tau)\\
&= \int_{G}\int_{\hat{G}} \overline{\gamma(t)}\alpha_{t^{-1}}(\Phi_1(x)(s,\gamma))W\xi(s^{-1}t,\gamma^{-1}\tau)d\nu(\gamma)d\mu(s)
  \\ &=\int_{G}\int_{\hat{G}}
  \overline{\gamma(t)}\alpha_{t^{-1}}(\gamma(s)x(\gamma,s))\overline{(\gamma^{-1}\tau)(s^{-1}t)}\xi(\gamma^{-1}\tau,s^{-1}t)d\nu(\gamma)d\mu(s)
\\&=\overline{\tau(t)}\int_{\hat{G}}\int_{G} \gamma(s)\alpha_{t^{-1}}(x(\tau,s))\xi(\gamma^{-1}\tau,s^{-1}t)d\mu(s)d\nu(\gamma)
\\&=\overline{\tau(t)}(x\xi)(\tau,t)=W(x\xi)(t,\tau).
\end{aligned}$$
This implies that $\Phi_1(x)=WxW^{-1}.$
For each $x\in W^*_p(\hat{G},W^*_p(G,A,\alpha),\hat{\alpha})$, we still define $\Phi_1(x)=WxW^{-1}.$

{\it Claim 1.} If $\{x_n\}$ is a sequence in $C_c(G\times\hat{G},A)$ such that $\{x_n\}$ converges to $x$ in weak* topology, then $\{\Phi_1(x_n)\}$ converges to $\Phi_1(x)\in W^*_p(G,W^*_p(\hat{G},A,\iota),\beta)$ in weak* topology.

Let $\{x_n\}$ be given as in Claim 1.
For each $\omega=\sum_{i=1}^\infty \zeta_i\otimes \theta_i\in L^q(G\times\hat{G},E')\widehat{\otimes}L^p(G\times\hat{G},E)$, then $\sum_{i=1}^\infty W'\zeta_i\otimes W^{-1}\theta_i\in L^q(\hat{G}\times G,E')\widehat{\otimes}L^p(\hat{G}\times G,E)$.
A direct computation shows that $$\begin{aligned} \la \Phi_1(x_n),\omega \ra&=
\sum_{i=1}^\infty \la \zeta_i, Wx_n W^{-1} \theta_i\ra
  \\ &=\sum_{i=1}^\infty\la W'\zeta_i,x_\alpha W^{-1} \theta_i\ra
\\&=\la x_n, \sum_{i=1}^\infty W'\zeta_i\otimes W^{-1}\theta_i\ra
\\&\rightarrow \la x, \sum_{i=1}^\infty W'\zeta_i\otimes W^{-1}\theta_i\ra
\\&=\la \Phi_1(x),\omega\ra .
\end{aligned}$$
This prove Claim 1.
Using similar method from Claim 1, one can check that $\Phi_1$ is weak* continuous. Obviously, $\Phi_1$ is isometric.
\end{proof}
%Since $C_c(G\times \hat{G},A)$ is weak* dense in $W^*_p(G,W^*_p(\hat{G},A,\iota),\beta)$ and $C_c(\hat{G}\times G,A)$ is weak* dense in $W^*_p(\hat{G},W^*_p(G,A,\alpha),\hat{G})$, it follows that $\Phi_1$ is a weak* continuous isometric isomorphism from $W^*_p(G,W^*_p(\hat{G},A,\iota),\beta)$ onto $W^*_p(\hat{G},W^*_p(G,A,\alpha),\hat{G})$.

\subsection{The Gelfand transform properties of $\Phi_2$}

In this subsection, we study the properties of $\phi_2$. This map is essentially  induced by Gelfand transform $\Gamma_p:PM_p(\hat{G})\rightarrow l^\infty(G)$. We will see that $\Phi_2$ share similar properties of $\Gamma_p$. To show the Gelfand transform is $p$-completely contractive, we prove the following lemma.
\begin{lem}\label{cb}
Let $A\subset \mathcal{B}(L^p(X,\mu))$ be a norm closed subalgebra, and let $B=L^\infty(Y,\nu)\subset \mathcal{B}(L^p(X,\nu))$. If $\varphi:A\rightarrow B$ is a bounded linear map, then $\|\varphi\|_{cb}=\|\varphi\|.$
\end{lem}
\begin{proof}
Since $L^p$-operator algebras are $p$-operator spaces, we will use the axioms (see \cite[axiom $\mathcal{M}_p$]{Daws1}) of $p$-operator spaces  to prove this Lemma.
Let $a=\sum_{i,j=1}^n e_{i,j}\otimes a_{i,j}\in M_n^p\otimes_p A$.
For each $\beta=\{\beta_1,\cdots,\beta_n\}^{T}\in \mathbb{C}^n$, we let $\beta^*=\{\beta_1,\cdots,\beta_n\},$ then we have the following
$$\begin{aligned} &\left\|\mathrm{id}_{M_n^p}\otimes \varphi(\sum_{i,j=1}^n e_{i,j}\otimes a_{i,j})\right\|
\\&=\left\|\sum_{i=1}^n(e_{i,j}\otimes \varphi(a_{i,j})\right\|
  \\ &=\sup_{y\in Y}\left\|\sum_{i,j=1}^n e_{i,j}\otimes \varphi(a_{i,j})(y)\right\|
\\&=\sup_{y\in Y,\alpha,\beta\in \mathbb{C}^n, \|\alpha\|_p=\|\beta\|_q=1}\left\la\beta,\left(\sum_{i,j=1}^n e_{i,j}\otimes \varphi(a_{i,j})(y)\right)\alpha\right\ra
\\&=\sup_{y\in Y,\alpha,\beta\in \mathbb{C}^n, \|\alpha\|_p=\|\beta\|_q=1}\left|\sum_{i,j=1}^n \beta_j\varphi(a_{i,j})(y)\alpha_i\right|
\\&=\sup_{y\in Y,\alpha,\beta\in \mathbb{C}^n, \|\alpha\|_p=\|\beta\|_q=1} \left|\varphi\left(\beta^*\left(\sum_{i,j=1}^ne_{i,j}\otimes a_{i,j}\right)\alpha\right)(y)\right|
\\&\leq \|\varphi\|\sup_{\alpha,\beta\in \mathbb{C}^n, \|\alpha\|_p=\|\beta\|_q=1} \left\|\beta^*\left(\sum_{i,j=1}^ne_{i,j}\otimes a_{i,j}\right)\alpha\right\|
\\&\leq \|\varphi\|\left\|\sum_{i,j=1}^ne_{i,j}\otimes a_{i,j}\right\|.
\end{aligned}$$
This shows that $\|\mathrm{id}_{M_n^p}\otimes \varphi\|\leq \|\varphi\|$ for all positive integer $n$, and thus $\|\mathrm{id}_{M_n^p}\otimes \varphi\|= \|\varphi\|$.
\end{proof}

Let $\mathrm{lt}\otimes\alpha$ be the action of $G$ on $l^\infty(G)\bar{\otimes}A$ given by $$(\mathrm{lt}\otimes\alpha)_s(f)(t)=\alpha_s(f(s^{-1}t))$$ for all $f\in l^\infty(G)\bar{\otimes}A$.
Recall that $\beta$ is the action of $G$ on $W^*_p(\hat{G},A,\iota)$ given by $$\left(\beta_t(f)\right)(\tau)=\tau(t)\alpha_t\left(f(\tau)\right)$$ for all $f\in C(\hat{G},A,\iota)$, $\tau\in\hat{G}$ and $t\in G$.

\begin{lem}\label{2}
Let $G$ be a countable discrete group.
Let $\varphi_2:=\Gamma_p\otimes \mathrm{id}_A:W^*_p(\hat{G},A,\iota)\rightarrow l^\infty(G)\bar{\otimes} A$. Then the following hold:
\begin{enumerate}
\item[(i)] the map $\varphi_2$ is weak* continuous $p$-completely contractive;
\item[(ii)]  the map $\varphi_2$ is injective with dense range in weak* topology;
\item[(iii)] the map $\varphi_2$ is surjective if and only if either $G$ is finite or $p=2$;
\item[(iv)] the map $\varphi_2$ is isometrically isomorphic if either $p=2$ or $G$ is trivial;
\item[(v)] $\varphi_2$ is equivariant for the action $\beta$ of $G$ on $W^*_p(\hat{G},A,\iota)$ and the action $\mathrm{lt}\otimes\alpha$ of $G$ on $l^\infty(G)\bar{\otimes} A$.
\end{enumerate}
\end{lem}
\begin{proof}
(i) Note that $\iota$ is the trivial action of $\hat{G}$ on $A$, then we have that $W^*_p(\hat{G},A,\iota)=PM_p(\hat{G})\bar{\otimes} A$.
Let $\Gamma_p:PM_p(\hat{G})\rightarrow L^\infty(G)$ be the Gelfand transform. By Lemma \ref{cb}, it follows that $\Gamma_p$ is $p$-completely contractive. Then by Lemma \ref{tensor}, there exists a weak* continuous $p$-completely contractive homomorphism $\hat{\Gamma}_p:PM_p(\hat{G})\bar{\otimes}A\rightarrow L^\infty(G)\bar{\otimes}A$ such that $\hat{\Gamma}_p(f\otimes a)=\Gamma_p(f)\otimes a$ for all $f\in PM_p(\hat{G})$ and $a\in A$. We denote by $\varphi_2=\Gamma_p\otimes \mathrm{id}_A=\hat{\Gamma}_p$. This proves (i).

(ii) Obviously, $\varphi_2$ has dense range.
Now we shall prove that $\varphi_2$ is injective.
For each $F\in PM_p(\hat{G})\bar{\otimes}A$, we assume that $\varphi_2(F)=0$. Let $F_n=\sum_{i=1}^{m(n)}f_i^{(n)}\otimes a_i^{(n)}\in PM_p(\hat{G})\otimes_{alg}A$ such that $\{F_n\}$ converges to $F$ in weak* topology. By uniform boundedness principle we may assume that $\{F_n\}$ is uniformly bounded. Therefore, by \cite[Theorem 6.3]{Daws1}, to prove $\{F_n\}$ converges to $0$ in weak* topology, it suffices to show that $$\lim_{n\rightarrow \infty}\omega\otimes\eta(F_n)=0,$$ where $\omega\in PM_p(\hat{G})_*$ and $\eta\in A_*$.
Note that the inclusion $A_2(G)\rightarrow A_p(G)$ is contractive with dense range (see \cite[Theorem C]{Herz} and \cite[Proposition 1.1]{Runde}),
then it suffices to show that $$\lim_{n\rightarrow \infty}(\sigma\circ\Gamma_p)\otimes\eta(F_n)=0,$$ where $\sigma\in l^\infty(G)_*\cong l^1(G)$.
Since $\varphi_2$ is weak* continuous, it follows that
$$\begin{aligned} (\sigma\circ\Gamma_p)(F_n) &=
\sum_{i=1}^{m(n)}\sum_{t\in G}\Gamma_p(f_i^{(n)})(t)\sigma(t)\eta(a_i^{(n)})\\
 &=\sigma\otimes\eta(\varphi_2(F_n))\rightarrow 0. \end{aligned}$$
This implies that $F=0$. Hence $\varphi_2$ is injective.

(iii) Let $p\neq 2$. By Theorem \ref{G},  it follows that $\Gamma_p$ is surjective if and only is $G$ is finite. Hence we have that $\varphi_2$ is surjective if and only if $G$ is finite.

(iv) If either $G$ is finite or $p=2$, it follows from Lemma \ref{tensor} that $\varphi_2$ is isometrically isomorphic.

(v) The Gelfand transformation $\Gamma_p:W^*_p(\hat{G})\rightarrow l^\infty(G)$ sends $f\in C(\hat{G})$ to its Fourier transform  $\hat{f}\in C_0(G)$ which is given by $$\hat{f}(t)=\int_{\hat{G}} f(\gamma)\overline{\gamma(t)}d\mu(\gamma).$$
Hence the homomorphism $\varphi_2=\Gamma_p\otimes \mathrm{id}_A:W^*_p(\hat{G},A,\iota)\rightarrow l^\infty(G)\bar{\otimes} A$ satisfies
$$\varphi_2(g)(t)=\int_{\hat{G}}g(\gamma)\overline{\gamma(t)}d\mu(\gamma)$$ for all $g\in C(\hat{G},A,\iota)$.
Since
$$\begin{aligned} \varphi_2(\beta_{s}(g))(t) &= \int_{\hat{G}}\beta_{s}(g)(\gamma)\overline{\gamma(t)}d\mu(\gamma) \\
 &=\int_{\hat{G}}\gamma(s)\alpha_{s}(g(\gamma))\overline{\gamma(t)}d\mu(\gamma)\\
 &=\alpha_{s}(\int_{\hat{G}}g(\gamma)\overline{\gamma(s^{-1}t)}d\mu(\gamma))\\
 &=(\mathrm{lt}\otimes\alpha)_{s}(\varphi_2(g))(t), \end{aligned}$$
and $\varphi_2$, $\beta_s$ and $(\mathrm{lt}\otimes\alpha)_{s}$ is weak* continuous,
it follows that $\varphi_2$ is equivariant.\end{proof}

By Lemma \ref{induce} and Lemma \ref{2}, there exists a weak* continuous contractive homomorphism $$\Phi_2:=\varphi_2\rtimes\mathrm{id}:W^*_p(G,W^*_p(\hat{G},A,\iota),\beta)\rightarrow W^*_p(G,l^\infty{G}\bar{\otimes} A,\mathrm{lt}\otimes\alpha).$$
\begin{prop} \label{22}
With the notations given above, then we have the following
\begin{enumerate}
\item[(i)] the map $\Phi_2$ is injective and has dense range in weak* topology;
\item[(ii)] the map $\Phi_2$ is surjective if and only if $G$ is finite or $p=2$;
\item [(iii)] the map $\Phi_2$ is isometrically isomorphic if either $p=2$ or $G$ is trivial.
\end{enumerate}
\end{prop}

\begin{proof}[Proof of Proposition \ref{2}]
(i) By Lemma \ref{exp}, we have the following
commutative diagram
$$\begin{CD}
  W^*_p(A,W^*_p(\hat{G},A,\iota),\beta) @ >\Phi_2>>W^*_p(G,l^\infty(G)\bar{\otimes} A,\mathrm{lt}\otimes\alpha) \\
     @VV E_s V @VV E_s V\\
     W^*_p(\hat{G},A,\iota) @>\varphi_2>> l^\infty(G)\bar{\otimes}A
\end{CD}.$$
Choose $F\in  W^*_p(A,W^*_p(\hat{G},A,\iota),\beta)$ such that $\Phi_2(F)=0$. Then $\varphi_2\circ E_s(F)=0$ for all $s\in G$.
Since $\varphi_2$ is injective by Lemma \ref{2}, it follows that $E_s(F)=0$ for all $s\in G$. This implies that $F=0$ by Lemma \ref{inj}.
Note that $\varphi_2$ has dense range and $\Phi_2$ is induced by $\varphi_2$, then so does $\Phi_2$. Obviously, $\Phi_2$ has dense range in weak* topology.

(ii) Let $p\in (1,\infty)\setminus\{2\}$. By Lemma \ref{2} (iii), it follows that $\Phi_2$ is surjective if and only if $G$ is finite.

(iii) If either $p=2$ or $G$ is trivial, it follows from Lemma \ref{induce} that $\Phi_2$ is isometrically isomorphic.
\end{proof}

\subsection{The weak* continuous isometry $\Phi_3$}

Let $\mathrm{id}$ be the trivial action of $G$ on $A$. Then we can form the tensor product action of $G$ on $l^\infty(G)\bar{\otimes}A$ by $(\mathrm{lt}\otimes\mathrm{id})_tf(s):=f(t^{-1}s)$.

\begin{prop}\label{3}
Let $\mathrm{lt}\otimes\mathrm{id}$ be defined as above. If, in addition, $\alpha$ is a $p$-completely isometric action of $G$ on $A$, then
there exists a weak* continuous isometric isomorphism $$\Phi_3:W^*_p(G,l^\infty(G)\bar{\otimes}A,\mathrm{lt}\otimes\alpha)\rightarrow W^*_p(G,l^\infty(G)\bar{\otimes}A,\mathrm{lt}\otimes\mathrm{id})$$ such that $$\Phi_3(F)(s,t)=\alpha_{t^{-1}}(F(s,t))$$ for all $F\in C_c(G,l^\infty(G)\bar{\otimes}A,\mathrm{lt}\otimes\alpha)$.
\end{prop}
\begin{proof}
Let $\varphi_3:l^\infty(G)\bar{\otimes}A\rightarrow l^\infty(G)\bar{\otimes}A$ be given by $$\varphi_3(f)(t)=\alpha_{t^{-1}}(f(t))$$ for all $f\in l^\infty(G)\bar{\otimes}A$. By Lemma \ref{induce}, it suffices to prove that $\varphi_3$ is a weak* continuous, $p$-completely isometric, equivariant map for the action $\mathrm{lt}\otimes\alpha$ of $G$ on $l^\infty(G)\bar{\otimes}A$ and the action $\mathrm{lt}\otimes\mathrm{id}$ of $G$ on $l^\infty(G)\bar{\otimes}A$.
Indeed, if we let $\Phi_3=\varphi_3\rtimes\mathrm{id}$, this proves Proposition \ref{3}.
Then we will prove these by showing the following three Claims.

{\it Claim 1.} $\varphi_3$ is weak* continuous.

Assume that $f_n,f\in l^\infty(G)\bar{\otimes}A$ such that $f=\lim_{\mathrm{weak}*,n}f_n$. For each $t\in G$, $\zeta\in L^q(X,\mu)$ and $\theta\in L^p(X,\mu)$, we have
$$\begin{aligned} \la \zeta, f_n(t)\theta\ra
 &=\la \delta_t\otimes \zeta, f_n(\delta_t\otimes\theta)\ra \\
 &\rightarrow \la \delta_t\otimes \zeta, f(\delta_t\otimes\theta)\ra\\
 &=\la \zeta,f(t)\theta\ra. \end{aligned}$$
This implies that $f_n$ is uniformly bounded, so there exists a positive number $M$ such that $\|f_n\|\leq M.$
Since $\alpha$ is a weak*continuous action of $G$ on $A$, it follows that $$\la\zeta,\alpha_{t^{-1}}(f_n(t))\theta\ra\rightarrow \la\zeta,\alpha_{t^{-1}}(f(t))\theta\ra.$$

Let $x=\sum_{i=1}^\infty \xi_i\otimes\eta_i\in (l^q(G)\otimes_qL^q(X,\mu))\widehat{\otimes}(l^p(G)\otimes_pL^p(X,\mu))$.
Since $$\begin{aligned} \sum_{i=1}^\infty\sum_{t\in G}|\la\xi_i(t),\alpha_{t^{-1}}(f_n(t))\eta_i(t)\ra|
 &\leq M \sum_{i=1}^\infty\sum_{t\in G}|\la\xi_i(t),\eta_i(t)\ra|\\
 &\leq M\sum_{i=1}^\infty \sum_{t\in G}\|\xi_i(t)\|_q\|\eta_i(t)\|_p\\
 &\leq M\sum_{i=1}^\infty (\sum_{t\in G}\|\xi_i(t)\|_{q}^q)^{\frac{1}{q}}(\sum_{t\in G}\|\eta_i(t)\|_{p}^p)^{\frac{1}{p}}\\
 &=M\sum_{i=1}^\infty \|\xi_i\|_q\|\eta_i\|_p<\infty, \end{aligned}$$
it follows from  Lebesgue's Dominated Convergence Theorem that
$$\begin{aligned} \la \varphi_3(f_n),x\ra
 &=\sum_{i=1}^\infty \la \xi_i,\varphi_3(f_n)\eta_i\ra \\
 &=\sum_{i=1}^\infty \sum_{t\in G}\la\xi_i(t),(\varphi_3(f_n)\eta_i(t))\ra\\
 &=\sum_{i=1}^\infty \sum_{t\in G}\la \xi_i(t),\alpha_{t^{-1}}(f_n(t))\eta_i(t)\ra \\
 &\rightarrow \sum_{i=1}^\infty \sum_{t\in G}\la\xi_i(t),\alpha_{t^{-1}}(f(t))\eta_i(t)\ra\\
 &=\sum_{i=1}^\infty \sum_{t\in G}\la\xi_i(t),\varphi_3(f)\eta_i(t)\ra\\
 &=\sum_{i=1}^\infty \la \xi_i,\varphi_3(f)\eta_i\ra\\
 &=\la \varphi_3(f),x\ra. \end{aligned}$$
This proves Claim 1.

{\it Claim 2.} $\varphi_3$ is $p$-completely isometric.

For each $\sum_{i,j=1}^n e_{i,j}\otimes f_{i,j}\in M_n^p\otimes_p A$, since $\alpha$ is a $p$-completely isometric action of $G$ on $A$, it follows that
$$\begin{aligned}\| \mathrm{id}_{M_n^p}\otimes\varphi_3(\sum_{i,j=1}^n e_{i,j}\otimes f_{i,j})\|
 &=\|\sum_{i,j=1}^n e_{i,j}\otimes \varphi_3(f_{i,j})\|\\
 &=\sup_{t\in G}\|\sum_{i,j=1}^ne_{i,j}\otimes \alpha_{t^{-1}}(f_{i,j}(t))\|\\
 &=\sup_{t\in G}\|\sum_{ij=1}^ne_{i,j}\otimes f_{i,j}(t)\|\\
 &=\|\sum_{i,j=1}^n e_{i,j}\otimes f_{i,j}\|. \end{aligned}$$
This proves Claim 2.

{\it Claim 3.} $\varphi_3$ is equivariant map for the action $\mathrm{lt}\otimes\alpha$ of $G$ on $l^\infty(G)\bar{\otimes}A$ and the action $\mathrm{lt}\otimes\mathrm{id}$ of $G$ on $l^\infty(G)\bar{\otimes}A$.

For each $f\in l^\infty\bar{\otimes}A$,
a direct computation shows that $$\begin{aligned} \varphi_{3}((\mathrm{lt}\otimes \alpha)_{s}(f))(t))&=
\alpha^{-1}_{t}(\alpha_{s}(f(s^{-1}t)))\\
&=\alpha^{-1}_{s^{-1}t}(f(s^{-1}t))\\
&=(\mathrm{lt}\otimes \mathrm{id})_{s}(\varphi_{3}(f))(t). \end{aligned}$$
This proves Claim 3.
\end{proof}

\subsection{The weak* continuous isometry $\Phi_4$}

\begin{prop}\label{4}
There exists a weak* continuous isometrical isomorphism $\Phi_4$ from $W^*_p(G,l^\infty(G)\bar{\otimes}A,\mathrm{lt}\otimes\mathrm{id})$ onto $\mathcal{B}(l^p(G))\bar{\otimes}A.$
\end{prop}
\begin{proof}
We define an invertible isometry $V:l^p(G)\otimes_p l^p(G)\otimes_p E\rightarrow l^p(G)\otimes_p l^p(G)\otimes_p E$ as $\delta_s\otimes\delta_t\otimes\xi\mapsto \delta_{st}\otimes \delta_{t}\otimes \xi$, where $\delta_s$ is the canonical basis of $l^p(G)$ for all $s\in G$ and $\xi\in E$.
We denote by $I_{l^p(G)}$ and $I_E$ the identity operator on $l^p(G)$ and $E$, respectively.

Let $\pi_0:l^\infty(G)\bar{\otimes}A\rightarrow \mathcal{B}(l^p(G)\otimes_pE)$ be a representation given by $\pi(f\otimes a)=M_f\otimes a$ for all $f\in l^\infty(G)$ and $a\in A$, where $M_f$ is the multiplication operator on $l^p(G)$. We denoted by $(\pi,\lambda_p^{l^p(G)\otimes_p E})$ be its associated regular representation.
A direct computation shows that $$\begin{aligned} V\pi(f\otimes a)(\delta_s\otimes\delta_t\otimes\xi)&=
V(\delta_s\otimes\pi_0((\mathrm{lt}\otimes\mathrm{id})_{s^{-1}}(f\otimes a))(\delta_t\otimes\xi))\\
&= V(\delta_s\otimes M_{(\mathrm{lt})_{s^{-1}}(f)}\delta_t\otimes a\xi))\\
&=V(\delta_s\otimes f(st)\delta_t\otimes a\xi)\\
&=f(st)\delta_{st}\otimes \delta_{t}\otimes a\xi\\
&=(M_f\otimes I_{l^p(G)}\otimes a)(\delta_{st}\otimes \delta_{t}\otimes \xi)\\
&=(M_f\otimes I_{l^p(G)}\otimes a)(V(\delta_s\otimes \delta_t\otimes\xi)). \end{aligned}$$
It follows that \begin{align}\label{pi}
V\pi(f\otimes a)V^{-1}= M_f\otimes I_{l^p(G)}\otimes a.
\end{align}

It is easy to check that $\lambda_p^{l^p(G,E)}=\lambda_p\otimes I_{l^p(G)}\otimes I_E$, where $\lambda_p:G\rightarrow \mathcal{B}(l^p(G))$ is the left-regular representation of $G$ on $l^p(G)$.
For each $r\in G$, one can check that  $$\begin{aligned} V(\lambda_p(r)\otimes I_{l^p(G)} \otimes I_E)V^{-1}(\delta_s\otimes\delta_t\otimes\xi)&=
 V(\lambda_p(r)\otimes I_{l^p(G)} \otimes I_E)(\delta_{st^{-1}}\otimes\delta_{t}\otimes\xi)\\
&=V(\delta_{rst^{-1}}\otimes\delta_t\otimes \xi)\\
&=\delta_{rs}\otimes \delta_{t}\otimes\xi\\
&=(\lambda_p(r)\otimes I_{l^p(G)}\otimes I_E)(\delta_s\otimes\delta_t\otimes\xi). \end{aligned}$$
It follows that \begin{align}\label{lambda}
V( \lambda_p(r)\otimes I_{l^p(G)}\otimes I_E)V^{-1}=\lambda_p(r)\otimes I_{l^p(G)}\otimes I_E.\end{align}

Since $f\in l^\infty(G)$ and $V$ is an isometry, by (\ref{pi}) and (\ref{lambda}), we have the following weak* continuous isometric isomorphism $$V(W^*_p(G,l^\infty(G)\bar{\otimes}A,\mathrm{lt}\otimes\mathrm{id}))V^{-1}\cong \mathcal{B}(l^p(G))\bar{\otimes} \mathbb{C}I_{l^p(G)}\bar{\otimes} A.$$
Obviously, $\mathcal{B}(l^p(G))\bar{\otimes} \mathbb{C}I_{l^p(G)}\bar{\otimes} A$ is weak* continuous isometrically isomorphic to $\mathcal{B}(l^p(G))\bar{\otimes}A$. We denote by $\Phi_4$ the weak* continuous isometric isomorphism from $W^*_p(G,l^\infty(G)\bar{\otimes}A,\mathrm{lt}\otimes\mathrm{id})$ onto $\mathcal{B}(l^p(G))\bar{\otimes}A$.
\end{proof}

Now we shall prove that $\Phi=\Phi_4\circ\Phi_3\circ\Phi_2\circ\Phi_1$ is equivariant.

For the reader's convenience, we recall the double dual action $\hat{\hat{\alpha}}$ of $G$ on $W^*_p(\hat{G},W^*_p(G,A,\alpha),\hat{\alpha})$ and the action $(\mathrm{rt}\otimes \alpha)\otimes \mathrm{id}$ of $G$ on $W^*_p(G,l^\infty(G)\bar{\otimes}A,\mathrm{lt}\otimes \mathrm{id})$. The double dual action $\hat{\hat{\alpha}}$ of $G$ on $C_{c}(\hat{G}\times G,A)\subset W^*_p(\hat{G},W^*_p(G,A,\alpha),\hat{\alpha})$ is given by $$\hat{\hat{\alpha}}_{t}(F)(\gamma, s):=\overline{\gamma(t)} F(\gamma, s)$$ for all $t\in G$. Let $\mathrm{rt}$ denote the right-translation of $G$ on $l^\infty(G)$, that is,
$(\mathrm{rt})_{t}f(s):=f(st)$ for all $f\in l^\infty(G)$. Then we get a weak* closed $L^{p}$-operator algebra dynamical system
$$\mathrm{rt}\otimes \alpha:G\rightarrow \mathrm{Aut}(l^\infty(G)\bar{\otimes}A),$$ where $$(\mathrm{rt}\otimes \alpha)_{t}f(s):=\alpha_{t}(f(st))$$ for all $f\in l^\infty(G)\bar{\otimes}A$ and $t\in G$.
Thus for each $F\in C_{c}(G\times G,A)\subset W^*_p(G,l^\infty(G)\bar{\otimes}A,\mathrm{lt}\otimes \mathrm{id})$ and $r\in G$, we let
$$((\mathrm{rt}\otimes \alpha)\otimes \mathrm{id})_{r}F(s,t):=\alpha_{r}(F(s,tr)).$$

Next let us introduce the action of $G$ on $\mathcal{B}(l^{p}(G))\otimes_{p}A$.
Let $\rho_p:G\rightarrow \mathcal{B}(l^{p}(G))$ be the right-regular representation and form the dynamical system $\mathrm{Ad}\rho_p: G\rightarrow \mathrm{Aut}(\mathcal{B}(l^{p}(G)))$, where $(\mathrm{Ad}\rho)_{s}T:=\rho_p(s)T\rho_p(s^{-1})$. Then we have the tensor product dynamical system $\mathrm{Ad}\rho)p\otimes\alpha:G\rightarrow \mathrm{Aut}(\mathcal{B}(l^{p}(G)))\otimes_{p}A$, where $(\mathrm{Ad}\rho_p\otimes\alpha)_s (T\otimes a)=(\mathrm{Ad}\rho_p)_s(T)\otimes \alpha_s(a)$.

\begin{prop}\label{equi}
Let $\Phi=\Phi_4\circ\Phi_3\circ\Phi_2\circ\Phi_1:W^*_p(\hat{G},W^*_p(G,A,\alpha),\hat{\alpha})\rightarrow \mathcal{B}(l^p(G))\bar{\otimes}A$. Then $\Phi$ is equivariant for the double dual of $\hat{\hat{\alpha}}$ of $G$ on $W^*_p(\hat{G},W^*_p(G,A,\alpha),\hat{\alpha})$ and the action $\mathrm{Ad}\rho_p\otimes\alpha$ of $G$ on $\mathcal{B}(l^p(G)\bar{\otimes}A$.
\end{prop}
\begin{proof}
Note that $C_c(\hat{G}\times G,A)$ is weak* dense in $W^*_p(\hat{G},W^*_p(G,A,\alpha),\hat{\alpha})$ and $\Phi$, $\hat{\hat{\alpha}}$, $\mathrm{Ad}\rho_p\otimes\alpha$ are weak* continuous, it suffice to prove that $$(\mathrm{Ad}\rho_p\otimes\alpha)_r\circ\Phi(F)=\Phi\circ\hat{\hat{\alpha}}(F)$$ for all $F\in C_c(\hat{G}\times G,A)$. Then we will prove this theorem by showing the following two Claims.

{\it Claim 1.} $(\Phi_{3}\circ \Phi_{2}\circ \Phi_{1})\circ \hat{\hat{\alpha}}_{r}(F)=((\mathrm{rt}\otimes\alpha)\otimes \mathrm{id})_{r}\circ (\Phi_{3}\circ \Phi_{2}\circ \Phi_{1})(F)$.

A direct computation shows that
$$\begin{aligned} \Phi_{3}\circ \Phi_{2}\circ \Phi_{1}(F)(s,t)&= \alpha_{t}^{-1}(\Phi_{2}\circ \Phi_{1}(F)(s,t))\\
&=\alpha_{t}^{-1}(\int_{\hat{G}}\Phi_{1}(F)(s,\gamma)\overline{\gamma(t)}d\mu(\gamma))\\
&=\int_{\hat{G}}\alpha_{t}^{-1}(F(\gamma,s))\overline{\gamma(s^{-1}t)}d\mu(\gamma). \end{aligned}$$
Then, for each $r\in G$, we have
$$\begin{aligned} \Phi_{3}\circ \Phi_{2}\circ \Phi_{1}(\hat{\hat{\alpha}}_{r}(F))(s,t)&= \int_{\hat{G}}\alpha_{t}^{-1}(\hat{\hat{\alpha}}_{r}(F)(\gamma,s))\overline{\gamma(s^{-1}t)}d\mu(\gamma)\\
&=\int_{\hat{G}}\alpha_{t}^{-1}(F(\gamma,s))\overline{\gamma(s^{-1}tr)}d\mu(\gamma)\\
&=\alpha_{r}(\int_{\hat{G}} \alpha_{tr}^{-1}(F(\gamma,s)\overline{\gamma(s^{-1}tr)}d\mu(\gamma))\\
&=\alpha_{r}(\Phi_{3}\circ \Phi_{2}\circ \Phi_{1}(F)(s,tr))\\
&=((\mathrm{rt}\otimes\alpha)\otimes \mathrm{id})_{r}\circ (\Phi_{3}\circ \Phi_{2}\circ \Phi_{1})(F)(s,t). \end{aligned}$$
This proves Claim 1.

{\bf Claim 2.} $(\mathrm{Ad}\rho_p\otimes\alpha)_{r}\circ\Phi_{4}(D)=\Phi_{4}\circ( (\mathrm{rt}\otimes\alpha)\otimes \mathrm{id})_{r}(D)$ for all $D\in C_c(G,l^\infty(G)\bar{\otimes}A,\mathrm{lt}\otimes \mathrm{id})\subset W^*_p(G,l^\infty(G)\bar{\otimes}A,\mathrm{lt}\otimes \mathrm{id}).$

For each $h\in C_{c}(G,E)\subseteq l^{p}(G,E)$, we have $\Phi_{4}(D)h(t)=\sum_{s\in G}D(s,t)h(s^{-1}t)$. For $r\in G$, $T\in \mathcal{B}(l^p(G))$ and $a\in A$, we have 
$$\begin{aligned} (\mathrm{Ad}\rho_p\otimes\alpha)_{r}(T\otimes a)&=
(\mathrm{Ad}\rho_p)_r(T)\otimes \alpha_r(a)\\
&=(\rho_p(r)\otimes I_E)((I_{l^p{G)}}\otimes\alpha_r)(T\otimes a))(\rho_p({r^{-1}})\otimes I_E).
 \end{aligned}$$
Hence
$$\begin{aligned} (\mathrm{Ad}\rho_p\otimes\alpha)_{r}(\Phi_{4}(D))h(t)&=
(\rho_p(r)\otimes I_E)((I_{l^p{G)}}\otimes\alpha_r)(\Phi_{4}(D)))(\rho_p({r^{-1}})\otimes I_E)h(t)\\
&=((I_{l^p{G)}}\otimes\alpha_r)(\Phi_{4}(D)))(\rho_p({r^{-1}})\otimes I_E)h(tr)\\
&=\sum_{s\in G}\alpha_r(D(s,tr))(\rho_p({r^{-1}})\otimes I_E)h(s^{-1}tr)\\
&=\sum_{s\in G}\alpha_r(D(s,tr))h(s^{-1}t)\\
&=\sum_{s\in G}((\mathrm{rt}\otimes\alpha)\otimes \mathrm{id})_{r}(D)(s,t)h(s^{-1}t)\\
&=\Phi_{4}( ((\mathrm{rt}\otimes\alpha)\otimes \mathrm{id})_{r}(D))h(t).
 \end{aligned}$$
This proves Claim 2.

By Claim 1 and Claim 2, one can check that $$(\mathrm{Ad}\rho_p\otimes\alpha)_r\circ\Phi(F)=\Phi\circ\hat{\hat{\alpha}}(F)$$ for all $F\in C_c(\hat{G}\times G,A)$.
\end{proof}

\section{Proof of Theorem \ref{Thm2} and \ref{Thm3}}

\subsection{Proof of Theorem \ref{Thm2}}
Note that isometric isomorphisms between $L^p$-operator algebras preserve $C^*$-cores, then Theorem \ref{Thm2} follows from proposition.
\begin{prop}\label{c1}
Let $p\in(1,\infty)\setminus\{2\}$, and let $G$ be a countable discrete group.
\begin{enumerate}
\item[(i)] $\mathrm{core}(W^*_p(G,W^*_p(\hat{G},A,\iota),\beta))=\mathrm{core}(A)$;
\item[(ii)]  $\mathrm{core}(\mathcal{B}(l^p(G))\bar{\otimes}(A))=l^\infty(G)\bar{\otimes}\mathrm{core}(A)$.
\end{enumerate}
\end{prop}
\begin{proof}
(i) By Proposition \ref{core1}, we have $\mathrm{core}(W^*_p(G,W^*_p(\hat{G},A,\iota),\beta))=\mathrm{core}(W^*_p(\hat{G},A,\iota))$.
Since $\iota$ is a trivial action of $\widehat{G}$ on $A$, it follows that $W^*_p(\hat{G},A,\iota)=PM_p(\widehat{G})\bar{\otimes}A$. Recall that $A\subset \mathcal{B}(L^p(X,\mu))$ is a weak* closed subalgebra, and $\nu$ is a normalized haar measure on $\widehat{G}$.
For each $f\in (PM_p(\hat{G}\bar{\otimes}A))_h$, then $f\in \mathcal{B}(L^p(\widehat{G}\times X,\nu\times \mu))_h$ (see \cite[Lemma 2.10]{Blecher and Phillips}).
Hence $f$ is a real-valued function in $L^\infty(\widehat{G}\times X,\nu\times \mu)$ (see \cite[Proposition 2.7]{Choi and}) and $f\in PM_p(\hat{G}\bar{\otimes}A)$. Hence $f(.,x)$ is a real-valued function in $L^\infty(\hat{G})$ and $f(.,x)\in PM_p(\widehat{G})$ for all $x\in X$.
Therefore,  $f(.,x)$ is a real-valued function in $L^\infty(\widehat{G})$ and $f(.,x)\in PM_p(\widehat{G})$.
It follows from Lemma \ref{core2} that $f(.,x)$ is a real-valued constant function in $L^\infty(\widehat{G})$, and we assume that $f(.,x)=c_x$, where $c_x(\gamma)=c_x\in \mathbb{R}$ for all $\gamma\in \widehat{G}$.

Similarly, for each $\gamma\in \widehat{G}$, we have $f(\gamma,.)\in A_h$. Then we have $$f(\gamma_1,x)=f(\gamma_2,x)$$ for all $\gamma_1,\gamma_2\in \widehat{G}$ and $x\in X$. Hence there is a bijective correspondence between $(PM_p(\hat{G})\bar{\otimes}A)_h$ and $A_h$ which is given by $f\mapsto f(\gamma,.)$. This shows that $\mathrm{core}(PM_p(\widehat{G})\bar{\otimes}A)=\mathrm{core}(A)$. This proves (i).

(ii) For each $f\in (\mathcal{B}(l^p(G))\bar{\otimes}A)_h$, then $f$ is a real-valued function in $L^p(G\times X,\nu\times\mu)$ and $f\in \mathcal{B}(l^p(G))\bar{\otimes}A$, where $\nu$ is the counting measure on $G$. Then $f(s)\in A_h$ for all $s\in G$. Hence $f=\sum_{s\in G}\delta_s\otimes f(s)\in l^\infty(G)_h\bar{\otimes} A_h$. Therefore,  $\mathrm{core}(\mathcal{B}(l^p(G))\bar{\otimes}(A))=l^\infty(G)\bar{\otimes}\mathrm{core}(A)$.
\end{proof}

\subsection{Proof Theorem \ref{Thm3}}
Note that isomorphisms between $L^p$-operator algebras preserve $p$-incompressibility, then Theorem \ref{Thm3} follows from the following proposition.
\begin{prop}
Let $p\in (1,\infty)$ and let $A=M_n^p$.
\begin{enumerate}
\item[(i)] $W^*_p(G,W^*_p(\hat{G},A,\iota),\beta))$ is $p$-incompressible if and only if $p=2$;
\item[(ii)]  $\mathcal{B}(l^p(G))\bar{\otimes} A$ is $p$-incompressible.
\end{enumerate}
\end{prop}
\begin{proof}
(i) Let $A=M_n^p$. If $p=2$, then $W^*_p(G,W^*_p(\hat{G},A,\iota),\beta))$ is a $C^*$-algebras. Therefore, $W^*_p(G,W^*_p(\hat{G},A,\iota),\beta))$ is $2$-incompressible. If $p\neq 2$, then $\Phi:W^*_p(G,W^*_p(\hat{G},A,\iota),\beta))\rightarrow \mathcal{B}(l^p(G))\bar{\otimes}A\subset\mathcal{B}(l_n^p\otimes_p l^p(G)))$ is injective contractive and can bounded below (see Theorem \ref{Thm1}). Hence $W^*_p(G,W^*_p(\hat{G},A,\iota),\beta))$ is not $p$-incompressible.

(ii) Since $\mathcal{B}(l^p(G))\bar{\otimes} A$ is isometrically isomorphic to $\mathcal{B}(l^p_n\otimes_p l^p(G))$, it follows from \cite{incompressible} that $\mathcal{B}(l^p(G))\bar{\otimes} A$ is $p$-incompressible.
\end{proof}

\section*{Acknowledgements}

The author is supported by the Research Start-up Funding Program of Hangzhou Normal University (grant number 4235C50224204070).

%We are much indebted to Yongjiang Duan, Kunyu Guo, Youqing Ji, Yanyue Shi, Kai Wang and Dechao Zheng for their interest and constructive comments on this work.
%The first author was supported by National Natural Science Foundation of China (Grant
%No. 11671167).

%%%%%%%%%%%%%%%%%%%%%%%%%%%%%%%%%%%%%%%%%55

\end{document}